\documentclass[a4paper,12pt]{amsart}

\usepackage{latexsym}
\usepackage{amssymb,amsmath}
\usepackage{url}
\newtheorem{teo}{Theorem}[section]
\newtheorem{corollary}[teo]{Corollary}

\newtheorem{proposition}[teo]{Proposition}

\newtheorem{lemma}[teo]{Lemma}
\newtheorem*{strategy}{Strategy}

\theoremstyle{definition}
\newtheorem{definition}[teo]{Definition}
\newtheorem{remark}[teo]{Remark}
\newtheorem{example}[teo]{Example}
\newtheorem{examples}[teo]{Examples}
\newtheorem{notation}[teo]{Notation}

\newtheorem*{First Main Result}{First Main Result}
\newtheorem*{Second Main Result}{Second Main Result}
\newcommand{\N}{\mathbb{N}}

\newcommand{\cO}{\mathcal{O}}
\newcommand{\cQ}{\mathcal{Q}}
\newcommand{\cG}{\mathcal{G}}
\newcommand{\F}{\mathcal{F}}
\newcommand{\T}{\mathcal{T}}

\newcommand{\vdasho}{\vdash_o }
\renewcommand{\l}{\ell}
\newcommand{\la}{\lambda}
\newcommand{\Qd}[2]{\cQ_2^{( #1 )}( #2 )}

\renewcommand{\epsilon}{\varepsilon}

\DeclareMathOperator{\Irr}{Irr}

\newcommand{\fS}{{\mathfrak{S}}}

\newcounter{thmlistcnt}
	{\setcounter{thmlistcnt}{0}%
	\begin{list}{\emph{(\roman{thmlistcnt})}}{%
		\usecounter{thmlistcnt}%
		\setlength{\topsep}{0pt}%
		\setlength{\leftmargin}{27pt}%
		\setlength{\itemsep}{0pt}%
		\setlength{\labelwidth}{20pt}
		\setlength{\itemindent}{0pt}}%
	}%
	{\end{list}}%

\title[]{Hook removal operators on the odd Young graph}
\author{Christine Bessenrodt, Eugenio Giannelli and J\o rn B. Olsson}

\address[C.~Bessenrodt]{Institute for Algebra, Number Theory and Discrete Mathematics, Leibniz Universit\"at Hannover, Welfengarten 1, D-30167 Hannover, Germany}
\email{bessen@math.uni-hannover.de}

\address[E.~Giannelli]{Department of Pure Mathematics and Mathematical Statistics, University of Cambridge, Cambridge CB3 0WA, UK}
\email{eg513@cam.ac.uk}

\address[J.\,B.~Olsson]{Department of Mathematical Sciences, University of Copenhagen, DK-2100 Copenhagen \O,
Denmark} \email{olsson@math.ku.dk}

\begin{document}

\thanks{The second author's research is funded by Trinity Hall, University of Cambridge.}

\begin{abstract}
In this article we consider hook removal operators on \textit{odd} partitions, i.e., partitions labelling odd-degree irreducible characters of finite symmetric groups. 
In particular we complete the discussion, started in \cite{INOT} and developed in \cite{BGO}, concerning the commutativity of such operators.
\end{abstract}

\maketitle

\section{Introduction}   \label{sec:intro}

Let $n$ be a natural number and let $\chi$ be an irreducible character of odd degree of the symmetric group $\fS_n$. In \cite{INOT} it was shown that whenever $k$ is a non-negative integer such that $2^k\leq n$, then there exists a unique irreducible constituent $f_k^n(\chi)$ of $\chi_{\fS_{n-2^k}}$ of odd degree
appearing with odd multiplicity. This result (together with \cite[Theorem 1.1]{G}) was the key ingredient to give a purely algebraic interpretation of the canonical McKay bijection for symmetric groups first described combinatorially in \cite[Theorem 4.3]{GKNT}. 

Let $\mathrm{Irr}_{2'}(\fS_n)$ be the set of irreducible characters of $\fS_n$ of odd degree. 
The main aim of this article is to continue the study, begun in \cite{BGO}, of the maps
$$f_k^n:\Irr_{2'}(\fS_n)\longrightarrow\Irr_{2'}(\fS_{n-2^k}),\ \ \text{for all}\ \ n,k\in\mathbb{N}.$$
To ease the notation we will sometimes write $f_k$ instead of $f^n_k$ when the domain of the function is clear from the context.

A first natural question in order to understand the \textit{coherence} of these functions with respect to each other is: \textit{when does $f^{n-2^k}_kf^n_k=f_{k+1}^n$ hold?}  (in brief:  $f_kf_k=f_{k+1}$?).
We are able to fully answer this question. In order to formulate the precise statements of our theorems we need to introduce some specific combinatorial concepts and notation. This is done in details in Section~\ref{sec: kk,k+1}. Here we restrict ourselves to summarise our main results.

\begin{First Main Result}
Let $n$ and $k$ be natural numbers such that $2^{k+1}\leq n$. Then:

\textbf{(i)} We completely describe the partitions labelling those irreducible characters $\chi$ such that $f^{n-2^k}_kf^n_k(\chi)=f_{k+1}^n(\chi)$. (This description is obtained in Corollary~\ref{cor:Gk-all} together with Lemma~\ref{lem: goodcores} and Proposition~\ref{prop: bases}.)

\textbf{(ii)} As a corollary of \textbf{(i)}, we give an explicit formula for the number of such irreducible characters. (This is contained in Theorem~\ref{theo: G} and Corollary~\ref{cor:G0-count}.)

\end{First Main Result}

 Surprisingly, we are also able to show that when $n/2^k$ is \textit{small},
 we only need to check the value of $\chi$ on one specified conjugacy class
 to see whether $f_kf_k(\chi)=f_{k+1}(\chi)$ or not.
In fact, if $\lfloor n/2^k \rfloor\le 4$ then there is a specified element $g\in\fS_n$,  depending on $n$ and $k$, such that for all $\chi \in \Irr_{2'}(\fS_n)$ we have that $f_kf_k(\chi)=f_{k+1}(\chi)$ if and only if $|\chi(g)|=1$ (Theorem~\ref{ThmA}).

In the second half of the article we focus on the commutativity problem.
More precisely, given $\ell, k, n\in\mathbb{N}$ such that $2^k+2^\ell\leq n$ we aim to determine the set $\mathcal{T}_{k,\ell}(n)$ consisting of those partitions of $n$ that label irreducible characters $\chi$ of $\fS_n$ such that $f_kf_\ell(\chi)=f_\ell f_k(\chi)$. Again, we postpone the description of the necessary combinatorial framework to Sections~\ref{sec:commutativity} and~\ref{sec:basecases}. Here we summarize our results.

\begin{Second Main Result}
Let $\ell, k, n\in\mathbb{N}$ be such that $2^k+2^\ell\leq n$. Then:

\textbf{(i)} We completely describe the partitions labelling the irreducible characters contained in $\mathcal{T}_{k,\ell}(n)$. (In Lemma~\ref{lem: reducedcount} we first reduce the problem to the case $k=0$. The case $k=0$ is fully analysed in Section~\ref{sec:basecases}.)

\textbf{(ii)} As a corollary of \textbf{(i)}, we give an explicit formula for $|\mathcal{T}_{k,\ell}(n)|$. (This is done in Theorem~\ref{theo: B} together with Corollary~\ref{cor: countfinal}.)
\end{Second Main Result}

We remark that our Second Main Result concludes the study of commutativity of the maps $f_k$. This work begun in \cite{INOT} and was subsequently developed in \cite{BGO}.
The characters of odd degrees on which we focus in this article are labelled by partitions which we call odd partitions; thus we focus on a subgraph of the Young graph which we call the odd Young graph. The hook removal operators that we discuss here act on this odd Young graph, and the main theme of this article is the relationship of the behavior of these operators on this graph.

\section{Background}\label{sec:background}

We start by recalling some facts in the representation theory of symmetric groups. We refer the reader to \cite{James}, \cite{JK} or \cite{OlssonBook} for a more detailed account.

\subsection{Partitions and hooks}

A partition $\lambda=(u_1,u_2,\dots,u_\ell)$ is a finite weakly decreasing sequence of positive integers. We say that $\lambda$ is a partition of $|\lambda|=\sum u_i$, written $\lambda\vdash|\lambda|$ and call the $u_i$'s the {\it parts} of $\la$.
The Young diagram of $\lambda$ is the set $$[\lambda]=\{(i,j)\in{\mathbb N}\times{\mathbb N}\mid 1\leq i\leq\ell,1\leq j\leq u_i\},$$
where we
think of the diagram in matrix orientation,
with the node $(1,1)$ in the upper left corner.
Given $(r,c)\in[\lambda]$, the corresponding {\it hook} $H_{(r,c)}(\lambda)$ is the set defined by $$H_{(r,c)}(\lambda)=\{(r,y)\in [\lambda]\ |\ y\geq c\}\cup\{(x,c)\in [\lambda]\ |\ x\geq r\}.$$
We let $h_{r,c}(\lambda)=|H_{(r,c)}(\lambda)|$ be the {\it length} of the hook, and we denote by $\mathcal{H}(\lambda)$ the set of hooks in $\lambda$.
When $h=H_{(r,c)}(\lambda)$ is a hook of $\lambda$, we say that we remove it from $\lambda$ by removing the nodes on the rim of the Young diagram,  starting at the hand $(r,u_r)$ and ending at the foot of column $c$; the partition thus obtained is denoted by $\lambda\setminus h$.
Let $e\in\mathbb{N}$.  A hook of length $e$ in $\la$ is called an  $e$-hook of $\lambda$ and
we denote by $\mathcal{H}_e(\lambda)$ the set of hooks in $\lambda$ whose lengths are divisible by $e$.

\subsection{Odd degree characters and odd partitions}

Let $n$ be a natural number. We let $\Irr(\fS_n)$ denote the set of irreducible characters of~$\fS_n$ and $\mathcal{P}(n)$ the set of partitions of $n$.
There is a natural correspondence $\lambda \leftrightarrow \chi^{\lambda} $ between $\mathcal{P}(n)$ and $\Irr(\fS_n)$. In this case we say then that $\lambda$ labels $\chi^{\lambda}$.
We denote by $\Irr_{2'}(\fS_n)$ the set of irreducible characters of $\fS_n$ of odd degree. If $\chi^{\lambda} \in \Irr_{2'}(\fS_n)$ we say that $\chi^{\lambda}$ is an {\it odd character},
we call $\lambda$ an {\it odd partition} of~$n$ and write $\lambda\vdash_o n$.
We denote by $\cO(n)$ the set of odd partitions of $n$; by convention,
$\cO(0)=\{\emptyset\}$.

The following key result was proved in \cite[Theorem A and Proposition 4.2]{INOT}, generalizing the beautiful result obtained in \cite[Theorem 1]{APS}.

\begin{teo}\label{theo: INOT}
Let $n,k$ be such that $2^k\le n$ and let $\lambda\vdash_o n$. Then there exists a unique $2^k$-hook $h$ in $\lambda$ such that the partition
$\mu=\lambda\setminus h$ obtained by removing $h$ from $\lambda$ is an odd partition of $n-2^k$. Furthermore, we then have $f_k^n(\chi^\lambda)=\chi^\mu$.
\end{teo}


In the setting of Theorem~\ref{theo: INOT} we abuse notation and consider $f_k^n$ also as a map between the corresponding sets of partitions $\mathcal{O}(n)$ and $\mathcal{O}(n-2^k)$. In particular we will write $f_k^n(\lambda)=\mu$.

\begin{definition}\label{def: odd-position}
Let $n\in \N$ and  $\lambda\vdash_o n$.
We call a hook $h$ in $\lambda$ {\it odd} if $\lambda\setminus h$ is again an odd partition.
If $2^k \le n$, we let $z_k(\la)$ and $s_k(\la)$ be the integers satisfying that $H_{z_k(\la),s_k(\la)}(\la)$ is the unique odd $2^k$-hook in $\la$. Also let $r_k(\la)=h_{z_k(\la),1}(\la)$, a first-column hook length.

\end{definition}

\subsection{Abacus configurations, $e$-cores, $e$-quotients}

As discovered by G. James, all of the operations on partitions concerning addition and removal of $e$-hooks  may conveniently be recorded on an $e$-{\it abacus}.
This is described in detail in \cite[Section 2.7]{JK}. Since we are going to use abaci in some proofs we introduce our modification of this important object, which may somehow be more flexible.

An $e$-abacus consists of an odd number of rows, say $2a+1$, labelled (from top to bottom) by the integers $-a,-a+1,\ldots, a-1,a$, and $e$ runners (or columns) labelled (from left to right) by $0,1, \ldots, e-1$. For any $(x,y)\in\{-a,-a+1,\ldots, a-1,a\}\times \{0,1, \ldots, e-1\}$ we can place a bead in position $(x,y)$ of the abacus.
Any distribution of beads on an $e$-abacus is called an {\it $e$-abacus configuration}.
Note that an $e$-abacus contains only finitely many beads in an $e$-abacus configuration. In all instances we choose the number of beads and rows large enough to perform all of the necessary operations.

A {\it gap} in an  $e$-abacus configuration $A$ is an empty position in $A$, i.e., a position not containing a bead.
Numbering the positions in $A$ from 0 to $e(2a+1)-1$, going from left to right and from top to bottom, we call the gap with the lowest position number the {\it first gap}.
Listing the numbers of the positions with beads in $A$ gives a set of non-negative integers, i.e., {\it a sequence of $\beta$-numbers} in the notation of \cite[Section 2.7]{JK}.
This is a generalization of a sequence of first-column hook lengths and uniquely identifies a partition $\lambda$.
In fact, if $c_1, c_2,\cdots,c_s$ are the beads on $A$ after the first gap, then the differences between the position numbers of these beads
and the position number of the first gap gives the first-column hook lengths $h_{(i,1)}(\la)$ of a unique partition $\la$. In this case we say that $A$ is an {\it $e$-abacus configuration} \textit{for} $\lambda$ (later we mostly call it simply an $e$-abacus  for $\lambda$).
Thus for a fixed $a$ the $e$-abacus configurations for a given $\lambda$ differ only by the number of beads before the first gap, i.e., by the position of the first gap. In particular, different $e$-abacus configurations for $\la$ have different numbers of beads.

We define a {\it normalized $e$-abacus configurations for $\la$} as one having its first gap in position $(0,0)$.

By \cite[2.7.13]{JK} the removal of an $e$-hook in $\la$ is recorded in any sequence of $\beta$-numbers for $\la$ by reducing one of its numbers by $e$.

We collect some of the above facts in a lemma for later reference.

\begin{lemma}\label{abacus-facts} Let $A$ be an $e$-abacus configuration for $\lambda$ consisting of $m$ beads.
\begin{enumerate}
\item[{(1)}] The positions of the beads after the first gap in $A$ determine the first-column hook lengths of $\la$.
Each of these beads contribute with a part of $\la$ which is equal to the number of gaps before the bead.

\item[{(2)}] The removal or addition of an $e$-hook is recorded on the $e$-abacus by sliding a bead one step up or down to an empty position on a runner of $A$.

\item[{(3)}] If  $A$ is normalized and if $|\la| < e$ then all beads after the first gap are in row $0$ of~$A$.
If there are $s$ such beads then their positions are $(0,h_{(i,1)}(\la))$ for $1 \le i \le s$.

\item[{(4)}] Let $B$ be an $e$-abacus configuration for $\lambda$. If $B$ consists of $m$ beads, then $A=B$.

\end{enumerate}
\end{lemma}

\begin{definition}\label{def: leftslide1}
Given an $e$-abacus configuration $A$ and a bead $c$ in $A$, we denote by $A(y, \leftarrow c)$ the $e$-abacus obtained by moving bead $c$ to its left $y$ times, ending in a gap. More precisely, the position number of $c$ in the abacus is reduced by $y$. Here we use \textit{left} in a broad sense. In fact, in principle we could end up in a gap in a higher row. (This will not happen in our applications later in the article).
\end{definition}

\begin{definition}\label{def: leftslide}
Let $n$ and $k$ be natural numbers such that $2^k\leq n$ and let $\lambda\vdasho n$.
Let $e$ be a natural number and let $A$ be an $e$-abacus configuration for $\lambda$. We denote by $c_{k}(A)$ the unique bead in $A$ such that $A(2^k, \leftarrow c_k(A))$ is an $e$-abacus configuration for $f_k(\lambda)$.
\end{definition}

We note that Theorem~\ref{theo: INOT} guarantees that the $c_{k}(A)$'s in  Definition~\ref{def: leftslide} are well defined.
The following is a consequence of Lemma~\ref{abacus-facts}(3).

\begin{lemma}\label{odd-hook-notations}
Let $2^k \le n < e$ and $\lambda\vdasho n$. Let $A$ be a normalized $e$-abacus configuration for $\lambda$. Then the bead $c_k(A)$ is in position $(0,r_k(A))$ for some $r_k(A) \in \{1, \ldots, n\}$. In fact, in the notation of Definition~\ref{def: odd-position} we have $r_k(A)=h_{z_k(\la),1}(\la)=r_k(\la)$, i.e., $r_k(A)$ is a
first-column hook length of $\la$.
\end{lemma}

Let $\lambda$ be a partition of $n$ and let $A$ be an $e$-abacus configuration for $\lambda$. Denote by $A_0$, $A_1$, $\ldots$, $A_{e-1}$ the runners in $A$ from left to right.  For $j\in\{0,\ldots, e-1\}$, denote by $|A_j|$ the number of beads on runner $j$. Moreover, we denote by $A^{\uparrow}$ the $e$-abacus obtained from $A$ by sliding all beads on each runner as high as possible. Similarly, we denote by $A_0^{\uparrow},\ldots, A_{e-1}^{\uparrow}$ the runners of $A^{\uparrow}$.
Then $A^{\uparrow}$ is an $e$-abacus for a partition $C_{e}(\lambda)$, known as the {\it $e$-core of $\la$}. Using Lemma~\ref{abacus-facts}(2) repeatedly, we see that $C_{e}(\lambda)$ does not contain an $e$-hook and is obtained from $\la$ be removing a number of $e$-hooks (say $w_e(\la)$ $e$-hooks) from~$\la$. We call $w_e(\la)$ the {\it $e$-weight} of~$\lambda$.

Then $w_e(\la)=w_e(A)$,  where $w_e(A)$ denotes the total number of up-slides needed to obtain $A^\uparrow$ from $A$. Similarly, for $i\in\{0,\ldots, e-1\}$ let $w_e(A_i)$ denote the number of those slides that were performed on runner $i$ in the transition from $A$ to $A^\uparrow$. Thus we have
 $$w_e(\lambda)=w_e(A)=w_e(A_0)+\ldots +w_e(A_{e-1}).$$
For $j\in\{0,\ldots, e-1\}$, denote by  $\mu_j=\mu_j(A)$  the partition corresponding to the runner~$A_j$ (considered as a $1$-abacus configuration). We then obviously have that $w_e(A_j)=|\mu_j|$.

\begin{remark}\label{rem: abacus}
Suppose that the $e$-abacus configuration  $A$ for the partition $\la$ of $n$ consists of $m$ beads.
The $e$-{\it quotient} $Q_{e}(\lambda)=(\lambda_0,\ldots, \lambda_{e-1})$ of $\la$, as defined in \cite[Section 2.7]{JK} and \cite[Section 3]{OlssonBook} agrees with $(\mu_0(A),\ldots, \mu_{e-1}(A))$ if $e \mid m$.  If $e \nmid m$ then  $(\mu_0,\ldots, \mu_{e-1})$ is just a $m'$-cyclic shift of $(\lambda_0,\ldots, \lambda_{e-1})$. Here $m'$ is the smallest positive integer, which is congruent to $m$ modulo $e$. In particular, $Q_{e}(\lambda)$ is also determined by $A$.
\end{remark}

 It is  well-known that a partition is uniquely determined by its $e$-core and $e$-quotient (see e.g.\ \cite[Proposition 3.7]{OlssonBook}). This fact may also be deduced from the discussion above.

\begin{remark}\label{compatible-hooks}
Let $\mathcal{H}(Q_e(\lambda))=\cup_{i=0}^{e-1} \mathcal{H}(\lambda_{i})$ (considered as multisets). As explained in \cite[Theorem 3.3]{OlssonBook}, there is a bijection between $\mathcal{H}_e(\lambda)$ and $\mathcal{H}(Q_e(\lambda))$ mapping hooks in $\lambda$ of length $ex$ to hooks in the quotient of length~$x$.
Moreover, the bijection respects the process of hook removal. In particular we conclude that
$$|\mathcal{H}_e(\lambda)|=w_e(\la) \; \text{ and } \; |\la|=|C_{e}(\lambda)|+ew_e(\la).$$
\end{remark}

\smallskip
For later use we also introduce the following object
when $e=2$.  We denote by $\mathcal{Q}_2(\lambda)$
the {\it $2$-quotient tower} and by $\mathcal{C}_2(\lambda)$ the {\it $2$-core tower}
of a partition~$\lambda$. These have rows numbered by $k \ge 0$.  The $k$th row  $\mathcal{Q}^{(k)}_2(\lambda)$ of $\mathcal{Q}_2(\lambda)$
contains $2^k$ partitions $\lambda^{(k)}_i$, $1 \le i \le 2^k$, and the $k$th row $\mathcal{C}^{(k)}_2(\lambda)$ of   $\mathcal{C}_2(\lambda)$ contains the $2$-cores of these partitions in the same order, i.e., $C_2(\lambda^{(k)}_i)$,  $1 \le i \le 2^k$. The 0th row of   $\mathcal{Q}_2(\lambda)$ contains $\lambda=\lambda^{(0)}_1$ itself,
row~1 contains the partitions $\lambda^{(1)}_1, \lambda^{(1)}_{2}$
occurring in the $2$-quotient $Q_2(\lambda)$, row~2 contains the partitions occurring in the $2$-quotients of partitions occurring in row~1,  and so on.
Specifically we have  $Q_2(\lambda^{(k)}_i)=(\lambda^{(k+1)}_{2i-1},\lambda^{(k+1)}_{2i})$ for $i \in \{1,\ldots,2^k\}$.
We remark that the $2^k$ partitions in $\mathcal{Q}^{(k)}_2(\lambda)$
are the same as those in the $2^k$-quotient $Q_{2^k}(\lambda)$ of $\lambda$, but in a different order for $k\ge 2$.
We also denote by $\mathcal{D}^{(k)}_2(\lambda)$ the {\em $k$-data} of~$\lambda$. This table (first introduced in \cite{BGO}) consists of $k+1$ rows.  For all $j\in\{0,\dots, k-1\}$ we have that row $j$ coincides with $\mathcal{C}^{(j)}_2(\lambda)$; the $(k+1)$th (last) row equals $\mathcal{Q}^{(k)}_2(\lambda)$.

\smallskip

\section{Preliminary results}\label{sec:prelims}

Having collected the combinatorial background material,
we now turn again to characters and prove some preliminary results that will
turn out to be useful in the later sections.

\subsection{Character values and odd partitions}

We apply the Murnaghan-Nakayama formula  (in short the MN-formula; see \cite [2.4.7]{JK}) for some easy but useful characterizations of odd partitions.

Suppose that $\chi^{\la} \in \Irr(\fS_n)$ and that $\pi \in \fS_n$ is a 2-element.
Then $\chi^{\la}(1)$ and $\chi^{\la}(\pi)$ have the same parity (see \cite[(6.4)]{Feit}).
In particular odd characters have odd values on all 2-elements,
and 
a character is odd
if and only if it takes odd values on some (non-empty) set of 2-elements.

The 2-powers in the binary expansion of numbers under consideration will play a crucial role.
When $2^a$ appears in such an expansion of $n$, we say that $2^a$ is a \textit{binary digit} of $n$ and write
$2^a\subseteq_2 n$.
Furthermore, we say that two natural numbers are \textit{$2$-disjoint} if they do not have any common binary digit.

We fix some notation for the following results.
We write $n=2^{a_1}+ \ldots + 2^{a_r}$ where $ a_1> \ldots >a_r \ge 0$.
Let $\omega_n$ be an element in $\fS_n$ of cycle type $(2^{a_1},2^{a_2}, \ldots ,2^{a_r})$.
A {\it 2-chain} in $\la \vdash n$ is a sequence $\la=\la^1,\ldots , \la^r$ where $\la^{i+1}$ is obtained from $\la^i$ by removing a $2^{a_i}$-hook.
If it exists, a 2-chain in $\la$ is unique.
We note that, as an immediate consequence of the main result in~\cite{B},
the number of partitions of $n$ that have a 2-chain is exactly $\prod_{i=1}^r 2^{a_i}$.

\smallskip

A part of the following result appeared in \cite[Lemma 1]{APS} with a different proof;
the criterion in part (3) will be applied repeatedly in later sections.

\begin{lemma}\label{lem:maxhook}
Let $\la \vdash n$. Let $\pi$ be a $2^{a_1}$-cycle in $\fS_n$.
\begin{enumerate}
\item[{(1)}]
If $\la$ contains a $2^{a_1}$-hook $h$, and
$\gamma \in \fS_n$ acts on the fixed points of $\pi$, then
$|\chi^{\la}(\pi \gamma)| =|\chi^{\la\setminus h}(\gamma)|$.

\item[{(2)}]
We have $\chi^{\la}(\pi) \ne 0$ if and only if $\la$ contains a $2^{a_1}$-hook $h$.
In that case $\chi^{\la}(1)$ and  $\chi^{\la\setminus h}(1)$ have the same parity.

\item[{(3)}] The partition $\la$ is odd if and only if $\la$ contains a $2^{a_1}$-hook $h$ such that $\la \setminus h$ is odd.

\item[{(4)}]\label{lem:chain-value}
If $\la$ contains a 2-chain then $|\chi^{\la}(\omega_n)|=1$.
\end{enumerate}
\end{lemma}
 \begin{proof}
(1) is an instance of the MN-formula.

(2) 
By (1), $\chi^{\la}(\pi) \neq 0$ if and only if
$\la$ contains a $2^{a_1}$-hook $h$.
If this is the case then
$\chi^{\la}(\pi) =\pm \chi^{\la\setminus h}(1)$,
so $\chi^{\la}(1)$ and  $\chi^{\la\setminus h}(1)$ have the same parity.

(3) If $\la$ is odd, then  $\chi^{\la}(\pi)$ is odd,
so that $h$ exists and $\la \setminus h$ is odd by (2).
If $h$ exists and $\la \setminus h$ is odd, then $\la$ is odd, by (2).

(4) is obtained by repeated application of (2).
\end{proof}

The next result is implicit in McKay's paper \cite{McKay}, where he introduced his famous conjecture.

\begin{lemma} \label{lem:odd-value}
Let $\la \vdash n$. The following statements are equivalent:
\begin{enumerate}
\item[{(i)}]  $\la$ is odd.

\item[{(ii)}]  $\la$ contains a 2-chain.

\item[{(iii)}]  $|\chi^{\la}(\omega_n)|=1$.
\end{enumerate}
\end{lemma}
\begin{proof}
(i) $\Leftrightarrow$ (ii) follows by repeated use of Lemma~\ref{lem:maxhook}(3).

(i) $\Rightarrow$ (iii) is stated in Lemma~\ref{lem:chain-value}(4), and the converse holds as $\chi^{\la}(\omega_n)$ and $\chi^\la(1)$ are of the same parity.
\end{proof}


\medskip

It is known that the Murnaghan-Nakayama formula may be iterated to get an interesting factorization formula for certain character values
involving $C_e(\lambda)$ and $Q_e(\la)$;  we recall this here (see \cite[2.7.33]{JK}).

\begin{lemma}\label{lem:farahat}
Let $\la\vdash n$, $Q_e(\la)=(\la_0,\ldots,\la_{e-1})$.
Assume that $w\ge w_e(\la)$, and let $\rho \in \fS_n$
be  the product of $w$ $e$-cycles;
let $\gamma \in \fS_n$ act on the fixed points of $\rho$.
Then
$$\chi^{\la}(\rho\gamma)=
\begin{cases}
\pm \binom{w}{|\la_0|,|\la_1|,\ldots,|\la_{e-1}|}\chi^{C_e(\la)}(\gamma)\prod_{i=0}^{e-1} \chi^{\la_i}(1) &  \text{if } w=w_e(\la) \\
 \rm{0} &  \text{if } w>w_e(\la)
\end{cases} \:.$$
\end{lemma}

For any $\la \vdash n$, $w_{2^k}(\la)  \le  \lfloor \frac{n}{2^k}\rfloor$, by Remark~\ref{compatible-hooks}.
If  $\la \vdasho n$ then $w_{2^k}(\la) =  \lfloor \frac{n}{2^k} \rfloor$,
by Theorem~\ref{theo: INOT}. 
Applying Lemma~\ref{lem:farahat}
with $e=2^k$ and $\rho$ a product of $ w=\lfloor \frac{n}{2^k} \rfloor$ $2^k$-cycles
we immediately obtain

\begin{corollary}\label{cor:farahat2}
Let $\la \vdash n$, with $\Qd{k}{\la}=(\la_1^{(k)}, \ldots, \la_{2^k}^{(k)})$.
If $\rho$  is a product of $ w=\lfloor \frac{n}{2^k} \rfloor$ $2^k$-cycles
and $\gamma \in \fS_n$ acts on the fixed points of $\rho$, then

$$\chi^{\la}(\rho \gamma)=
\begin{cases}
\pm \binom{w}{|\lambda^{(k)}_{1}|,|\lambda^{(k)}_{2}|,\ldots ,|\lambda^{(k)}_{2^k}|}\chi^{C_{2^k}(\la)}(\gamma)\prod_{i=1}^{2^k} \chi^{\lambda^{(k)}_{i}}(1)  & \text{if } w=w_{2^k}(\la)  \\
0 & \text{if } w>w_{2^k}(\la)
\end{cases}\:.$$
\end{corollary}
\smallskip

\begin{corollary} \label{cor:farahat}
Let $\la\vdash n$, $\rho,\gamma$ be as in Corollary~\ref{cor:farahat2}.
Then
$|\chi^{\la}(\rho \gamma)|=1$ if and only if the following conditions are all fulfilled:
\begin{itemize}
\item[(i)] There exists a unique $s \in \{1,2,\ldots,2^k \}$ such that  $\lambda^{(k)}_s \in \{(w),(1^w)\}$.
\item[(ii)] $|\chi^{C_{2^k}(\la)}(\gamma)| =1$.
\item[(iii)] $w=w_{2^k}(\la) $.
\end{itemize}
\end{corollary}

As a consequence of the previous results
we find all the character values $\pm1$ on certain 2-elements.
We use the notation introduced before;
in particular, $n=2^{a_1}+ \ldots + 2^{a_r}$ where $ a_1> \ldots >a_r \ge 0$.

\begin{corollary} \label{cor:value1}
Let $k\in \N$ with $2^k\le n$.
Set $w=\lfloor n/2^k \rfloor$, $m=n-w\cdot 2^k$.
Let $\rho\in\fS_n$ be a product of $w$ $2^k$-cycles,
$\gamma=\omega_m$ on the fixed points of $\rho$.
Let $\la\vdasho n$.
Then
$|\chi^{\la}(\rho \gamma)|=1$ if and only if  $\lambda^{(k)}_s\in \{(w),(1^w)\}$
for some $s \in \{1,2,\ldots,2^k \}$.

In particular, when $a_q\ge k > a_{q+1}$, there are exactly
$2^{k'}\prod_{i=q+1}^r 2^{a_i}$
irreducible characters with value $\pm 1$ on $\rho\gamma$,
where $k'=k$ when $w=1$, and $k'=k+1$ when $w>1$.
\end{corollary}

\begin{proof}
Since $\la\vdasho n$,
$w_{2^k}(\la) =  \lfloor \frac{n}{2^k} \rfloor$.
Furthermore, $C_{2^k}(\la)$ is an odd partition of $m=n-w\cdot 2^k$. Hence
$|C_{2^k}(\la)(\omega_m)|=1$ by Lemma~\ref{lem:odd-value}.
Thus both conditions (ii) and (iii) of Corollary~\ref{cor:farahat} are satisfied
for $\la$, and the criterion for
$|\chi^{\la}(\rho \gamma)|=1$ follows from Corollary~\ref{cor:farahat}; note that
the index $s$ is unique as the total weight of $\Qd{k}{\la}$ is $w$.

For counting the $\la\vdasho n$ satisfying $|\chi^{\la}(\rho \gamma)|=1$, note first that
there are exactly $\prod_{i=q+1}^r 2^{a_i}$ possible partitions $C_{2^k}(\la)$;
then there are $2^k$ possible indices $s$
such that $\lambda^{(k)}_s\in \{(w),(1^w)\}$, and either $1$ or $2$ possibilities
for   $\lambda^{(k)}_s$, when $w=1$ or $w>1$, respectively.
\end{proof}

For $\rho$ a product of $w_{2^k}(\la)$ $2^k$-cycles,
we know that $\la$ is odd if and only if  $\chi^{\la}(\rho)$ is odd.
This happens if and only if the following numbers are all odd:
$$\binom{w}{|\lambda^{(k)}_{1}|, |\lambda^{(k)}_{2}|,\ldots ,|\lambda^{(k)}_{2^k}|}, \quad
 \chi^{\lambda^{(k)}_{i}}(1) \text{, for all } i, \quad  \chi^{C_{2^k}(\la)}(1)\:.$$
Now
$\binom{w}{|\lambda^{(k)}_{1}|, |\lambda^{(k)}_{2}|,\ldots ,|\lambda^{(k)}_{2^k}|}$
is odd if and only if  the numbers $|\lambda^{(k)}_i|$,  $1\le i \le 2^k$,  are pairwise 2-disjoint.
Thus  $\lambda$ is odd if and only if the following conditions are all fulfilled:
\begin{itemize}
 \item[(i)] $C_{2^k}(\la)$ is odd.
\item[(ii)]  The partitions $\lambda_i^{(k)}$, $1\le i \le 2^k$, are all odd.
\item[(iii)] The numbers $|\lambda^{(k)}_i|$,  $1\le i \le 2^k$,  are pairwise 2-disjoint.
\end{itemize}

From this we may deduce the following theorem, which is also \cite[Theorem 2.5]{BGO}.
It shows that the $k$-data encodes full information on the parity of the character degree $\chi^\lambda(1)$.

\begin{teo}\label{thm:oddcriterion} 
Let $\lambda \vdash n$, and let $k \ge 0$ be fixed.
Consider
$\mathcal{Q}_2^{(k)}(\lambda)=(\la_1^{(k)},\ldots,\la_{2^k}^{(k)})$.
Then $\lambda$ is odd if and only if the following conditions are all fulfilled:
\begin{itemize}
\item[(i)]  $c_2^{(j)}(\lambda) \leq 1$ for all $j < k$.
\item[(ii)]  The partitions $\lambda_i^{(k)}$, $1\le i \le 2^k$, are all odd.
\item[(iii)] The numbers $|\lambda^{(k)}_i|$,  $1\le i \le 2^k$,  are pairwise 2-disjoint.
\end{itemize}
In this case  $ \sum_{i \ge 1} |\lambda^{(k)}_i|=\lfloor \frac{n}{2^k} \rfloor$.
\end{teo}

\smallskip

\subsection{Odd partitions: $R$-operators, $k$-types}

Let $n, j\in\mathbb{N}$ be such that $2\leq 2^j\leq n$. Let $\lambda\vdash_o n$ and let $h=H_{z_j(\la),s_j(\la)}(\la)$ be the unique $2^j$-hook in~$\lambda$ whose removal leads to the odd partition $f_j^n(\lambda)$ of $n-2^j$ (see Definition~\ref{def: odd-position}).
Repeated applications of Remark~\ref{compatible-hooks}  show that for any $k\in\{0,1,\ldots, j\}$ the removal of $h$ from $\lambda$ corresponds to the removal of a $2^{j-k}$-hook $\overline{h}$ from exactly one of the $2^{k}$ partitions involved in $\Qd{k}{\la}$.
We denote this distinguished component of $\Qd{k}{\la}$
by $R_k^{j}(\lambda)$; note that this partition occurs only once in
$\Qd{k}{\la}$, as for an odd partition
$\lambda$ the nonempty partitions in $\Qd{k}{\la}$ are
of different sizes (by Theorem~\ref{thm:oddcriterion} (iii)).
It is easy to see that for all $k\in\{0,1,\ldots, j\}$ the $k$-data of $f_j^n(\lambda)$ is obtained from the $k$-data of $\lambda$ by replacing $R_k^j(\lambda)$ with
$f_{j-k}^{r}(R_k^{j}(\lambda))$, where $r=|R_k^{j}(\lambda)|$.

\smallskip

We denote by $\mathcal{H}(n)$ the set of hook partitions of~$n$. The following lemma shows that if $\lambda$ belongs to a specific class of partitions, then it is very easy to explicitly describe $f_k(\lambda)$.

\begin{lemma}\label{lem: hooks}
Let $t,k, w\in\mathbb{N}$ be such that $0\le k < t$ and $0\leq w\leq 2^t-1$,
and let $\lambda=(2^t-w,1^w)\in\mathcal{H}(2^t)$. Then
$$f_k(\lambda)= \begin{cases}
(2^t-w,1^{w-2^k}) & \mathrm{if}\ 2^k \subseteq_2 w,\\
(2^t-w-2^k,1^w) & \mathrm{if}\ 2^k \not\subseteq_2 w.
\end{cases}$$
\end{lemma}
\begin{proof}
It is well known that if $\mu=(n-w, 1^w)\in\mathcal{H}(n)$, then $\chi^\mu(1)={n-1 \choose w}$.
The statement of the lemma follows directly from this fact.
\end{proof}

\smallskip

For arbitrary partitions $\lambda$ of an arbitrary natural number $n$, it is not always as easy as it is for hook partitions of a power of $2$ to understand the structure of $f_k^n(\lambda)$. For this reason
we give an explicit arithmetic determination for $R_k^j(\la)$ which will turn out to be very useful.

\begin{lemma}\label{lem:R-arithmetic}
Let
$n=2^{a_1} + \ldots + 2^{a_r}$ with
$a_1>a_2 > \ldots > a_r\ge 0$.
Let $\la\vdasho n$
and $j\in \N$ with $2^j\le n$.
Let $k\in \N_0$ with $k\le j$ and
$
\Qd{k}{\la} =(\lambda_1,\ldots, \lambda_{2^k})$.
Let $a_t=\min\{a_i\mid a_i\ge j\}$.
Then $2^{a_t-k}$ is a binary digit of $|\la_s|$ for a unique
$s\in \{1,\ldots,2^k\}$,
and $R_k^j(\la)=\la_s$.
\end{lemma}

\begin{proof}
By Theorem~\ref{thm:oddcriterion} we know that $\sum_{i=1}^{2^k} |\lambda_i|=\lfloor \frac{n}{2^k} \rfloor$ and that the numbers $|\lambda_i|$ are pairwise $2$-disjoint. Hence there exists a unique $s\in\{1,\ldots, 2^k\}$ such that $2^{a_t-k}$ is a binary digit of $|\la_s|$. Moreover, by definition we have that $\lambda_s$ is an odd partition such that $2^{j-k}\leq |\la_s|$. Hence $f_{j-k}(\lambda_s)$ is a well defined odd partition whose size is $2$-disjoint from $|\lambda_r|$ for all $r\in\{1,\ldots,2^k\}\setminus\{s\}$. Thus $R_k^j(\la)=\la_s$.
\end{proof}

In particular, for the realization of the removal of a $2^k$-hook from $\la$ in its $k$-data,
we look for the partition $\la_s$
such that the smallest binary digit of $\lfloor n/2^k\rfloor$ is a digit of $|\lambda_s|$.

\medskip

Because of this arithmetic description,
the following definition will play an important role.

\begin{definition}
Let $n,k \in \N$, $2^k\le n$.
Let $\la \vdash n$, with
$\Qd{k}{\la}=(\la_1,\ldots,\la_{2^k})$.
We say that $\la$ is
of {\em $k$-type}
$(n_1,\ldots,n_{2^k})$ (of size $\lfloor n/2^k \rfloor$)
if $n_j=|\la_j|$ for $1\le j \le 2^k$.

We let $\cO(n;\tau)$
denote the subset of $\cO(n)$
of partitions of $k$-type $\tau$.
\end{definition}

Note that for odd partitions of $n$,
their $k$-types are exactly the
2-disjoint weak compositions
of $\lfloor n/2^k\rfloor$ of length $2^k$.
Moreover, for any $j\ge k$, the position of
$R_k^j(\la)$ only depends on the $k$-type of $\la$, by Lemma~\ref{lem:R-arithmetic}.


\section{Comparing $f_k f_k$  with $f_{k+1}$}\label{sec: kk,k+1}

Let $n$ and $k$ be natural numbers such that $2^{k+1}\leq n$.
The aim of this section is to completely describe those partitions $\lambda\in\mathcal{O}(n)$ such that $f_kf_k(\lambda)=f_{k+1}(\lambda)$.
The strategy we adopt to achieve our goal is to reduce our problem to the study of the case where $k=0$.

We define
$$
\cG_k(n)=\{\la \vdasho n \mid f_kf_k(\la)=f_{k+1}(\la)\}
$$
and set $G_k(n)=|\cG_k(n)|$.
For a $k$-type $\tau$ of size $\lfloor n/2^k \rfloor$, we let $\cG_k(n;\tau)$ be the subset of $\cG_k(n)$
of partitions of $k$-type~$\tau$, and set $G_k(n;\tau)=|\cG_k(n;\tau)|$.
Furthermore, we set
$\displaystyle
\cG_k=\bigcup_{n\ge 2^{k+1}} \cG_k(n)\:.
$

\medskip

The reduction to the case $k=0$ is possible thanks to the following lemma.

\begin{lemma}\label{lem: fk2-fk+1}
Let $k, n\in\mathbb{N}_0$ be such that $2^{k+1}\leq n$. Let $\lambda\vdasho n$.
Then
$\la \in \cG_k$ if and only if
$R_k^{k}(\lambda)= R_k^{k+1}(\lambda) \in \cG_0$.
\end{lemma}
\begin{proof}
Suppose that
$R_k^{k}(\lambda)= R_k^{k+1}(\lambda) =:\mu \in \cG_0$.
Then
$R_k^{k}(f_k(\lambda))=f_0(R_k^{k+1}(\lambda))=f_0(\mu)$. Hence the $k$-data of $f_kf_k(\lambda)$ and  $f_{k+1}(\lambda)$, respectively, are obtained
from that of $\la$ by replacing $\mu$ with  $f_0f_0(\mu)$ and $f_1(\mu)$, respectively.
Since $\mu \in \cG_0$, these are equal, 
and since the $k$-data uniquely determines the partition, we deduce that $f_kf_k(\lambda)=f_{k+1}(\lambda)$.

In order to prove the converse,
we assume that $f_kf_k(\lambda)=f_{k+1}(\lambda)$.
Then
$R_k^{k}(\lambda)= R_k^{k+1}(\lambda)=:\mu$ as
otherwise $\Qd{k}{f_kf_k(\lambda)} \ne \Qd{k}{f_{k+1}(\lambda)}$.
Furthermore the $k$-data of $f_kf_k(\lambda)$ and  $f_{k+1}(\lambda)$, respectively, are obtained
from that of $\la$ by replacing $\mu$
with  $f_0f_0(\mu)$ and $f_1(\mu)$, respectively.
Hence we deduce  $\mu \in \cG_0$.
\end{proof}

We fix below some notation for the preparatory results
towards the counting formula in Theorem~\ref{theo: G}.

\begin{notation}\label{not:33}
Let $n,k\in \N_0$ be such that $2^{k+1}\leq n$.
We write
$n=2^{a_1} + \ldots + 2^{a_r}$, where
$a_1>a_2 > \ldots > a_r$.
We define $p,q \in\{1,\ldots, r\}$
to be maximal such that $a_p\ge k+1$ and $a_q\ge k$, respectively.
Set $b_j=a_j-k$, for $1\le j \le q$, and
$B=\{b_1,\ldots,b_q\}$.
Note that then either $b_p>b_q=0$ or $b_p=b_q > 0$.
\end{notation}

Using the definition and notation above,
we may reformulate Lemma~\ref{lem: fk2-fk+1} as follows;
this is a direct consequence of  Lemma~\ref{lem:R-arithmetic}.

\begin{corollary}\label{cor: fk2-fk+1}
Let  $\lambda\in \cO(n;\tau)$ with
$\Qd{k}{\la}=(\la_1,\ldots,\la_{2^k})$,
$\tau=(n_1,\ldots,n_{2^k})$.
Then $\lambda\in \cG_k(n;\tau)$ if and only if
there is an $s\in \{1,\ldots,2^k\}$ such that
$2^{b_p},2^{b_q}\subseteq_2 n_s$ and
$\la_s\in \cG_0(n_s)$.
\end{corollary}

We also immediately deduce the following description.

\begin{lemma}\label{lem: reduced-ff-count}
Let $\tau=(n_1,\ldots,n_{2^k})$ be a $k$-type of size $\lfloor n/2^k\rfloor$.
\begin{enumerate}
\item
If $2^{b_p},2^{b_q}$ are binary digits of two different parts of $\tau$,
then $\cG_k(n,\tau)=\emptyset$.
\item
If $2^{b_p},2^{b_q}$ are binary digits of the same part of $\tau$, say $n_j$,
then
$$
\cG_k(n;\tau)=
\{\la \in \cO(n;\tau) \mid
\Qd{k}{\la}_j \in \cG_0(n_j)
\}.
$$
\end{enumerate}
\end{lemma}

From the results above we deduce that for $\la\in \cO(n)$,
the property $f_k f_k(\la)=f_{k+1}(\la)$
does not depend on the rows of the $k$-data corresponding to the cores.
Hence towards counting these odd partitions we have:

\begin{lemma}\label{lem: ff-core-part}
We set $\bar n= \sum_{i=1}^q 2^{a_i}$.
For any $k$-type $\tau$ of size $\lfloor n/2^k\rfloor$ we have
$$
G_k(n;\tau)=(\prod_{j=q+1}^r 2^{a_j})\cdot G_k(\bar n;\tau).
$$
In particular,
$$
G_k(n)=(\prod_{j=q+1}^r 2^{a_j})\cdot G_k(\bar n)\:.
$$
\end{lemma}

To state the reduction formula for $G_k(n)$ in a nice way,
we introduce weights that are based on the counts for the case $k=0$.
For the numbers $G_0(n)$, we will derive explicit formulae later in this section.
(A similar approach will be taken for the commutativity problem in later sections.)

\begin{definition}
We fix $k\in\N$.
For $J\subseteq I=\{1,\ldots,q\}$
we define its $G$-weight with respect to the set
$B=\{b_1,\ldots,b_q\}$ to be
$$
w_G(J) =
(\prod_{i\in I\setminus J}2^{b_i})
\cdot
(2^k-1)^{|I\setminus J|}
\cdot
G_0(\sum_{j\in J}2^{b_j}).
$$
\end{definition}

Then we have the following formula for the number $G_k(n)$, based on the values $G_0(u)$.

\begin{teo}\label{theo: G}
With the notation introduced in~\ref{not:33}, we have
$$
G_k(n)=
(\prod_{j=q+1}^r 2^{a_j})\cdot
2^k \cdot
\sum_{\{p,q\}\subseteq J \subseteq I} w_G(J) \:.
$$
\end{teo}

\begin{proof}
In view of Lemma~\ref{lem: ff-core-part}, we focus on
computing the numbers $G_k(\bar n)$.
We take the sum over all $k$-types of size $\lfloor n/2^k\rfloor$,
and for each $k$-type $\tau$
we count the partitions in $\cG_k(\bar n;\tau)$
according to Lemma~\ref{lem: reduced-ff-count}.

This lemma tells us that we only get a contribution from
those $k$-types $\tau$  where $2^{b_p},2^{b_q}$
are binary digits in the same part of $\tau$.
To construct these,
take any $J\subseteq I$
containing $p,q$ and define $n_J:=\sum_{j\in J}2^{b_j}$.
We now consider the contributions from all $k$-types for this fixed~$J$.
We first have $2^k$ choices
for the position $j$ of the part
of which $2^{b_p},2^{b_q}$ are
binary digits.
For each $i\in I\setminus J$
we have $2^k-1$ choices to place $2^{b_i}$
as a binary digit of a part of the type $\tau$ under construction, different from the $j$-th.
By Lemma~\ref{lem: reduced-ff-count}, for such a $\tau=(n_1,\ldots,n_{2^k})$,
the partitions $\la\in \cG_k(\bar n;\tau)$
are then those
such that $\Qd{k}{\la}_i$ is an arbitrary
odd partition of $n_i$ for $i\ne j$,
and $\Qd{k}{\la}_j\in \cG_0(n_j)$
(and the earlier rows of the $k$-data contain only empty cores).
Hence we get
$|\cG_k(\bar n;\tau)|=(\prod_{i\in I\setminus J}2^{b_i})
\cdot
G_0(\sum_{j\in J}2^{b_j})$,
and the total contribution for $J$ is exactly $2^k w_G(J)$.

This proves the stated formula.
\end{proof}
In light of the reduction results obtained so far, we now turn to the case where $k=0$.

The following lemma will be used several times and
utilizes
some of the concepts introduced in Sections~\ref{sec:background} and~\ref{sec:prelims}. It is a generalization of \cite[Main Lemma]{APS}, and may also be found in \cite{INOT} with a different proof.

\begin{lemma}\label{lem: f commutes with core}
Let $n=2^t+m$ with $m<2^t$,  and let $\ell\in\mathbb{N}_0$ be such that $2^\ell\leq m$. If $\lambda\vdasho n$ then $f_{\ell}(C_{2^t}(\lambda))=C_{2^t}(f_\ell(\lambda))$.
\end{lemma}
\begin{proof}
Let $B$ be a normalized $2^t$-abacus for $\gamma=C_{2^t}(\lambda)$.
By Remark~\ref{compatible-hooks}, we have that $\la$ contains a unique $2^t$-hook and thus $\gamma \vdasho m$ by Theorem \ref{theo: INOT}.
Using Definition~\ref{def: leftslide}
we see that if $c=c_{\ell}(B)$ then $B(2^\ell, \leftarrow c)$ is a $2^t$-abacus configuration for $f_\ell(\gamma)$.
Lemma~\ref{odd-hook-notations} with $k=\ell$ and $m<e=2^t$ shows that $c$ is in position $(0,r_{\ell}(B))$ where
$r_\ell(\gamma)=r_\ell(B) \in \{2^{\ell},\ldots, 2^t-1\}$.

Let now $A$ be the abacus configuration for $\lambda$ such that $A^{\uparrow}=B$.
By Lemma~\ref{abacus-facts}, $A$ is obtained from $B$ by sliding one bead down one row; in particular $w_{2^t}(A)=1$.
Let $x\in\{0,1,\ldots, 2^{t}-1\}$ be such that $w_{2^t}(A_x)=1$ and $w_{2^t}(A_y)=0$ for all $y\neq x$.
Let $d=c_{\ell}(A)$.
Thus  $A(2^\ell, \leftarrow d)$ is a $2^t$-abacus for $f_\ell(\lambda)$. Put $j:=r_{\ell}(B)$. 
If $x\notin\{j, j-2^\ell\}$ then we have that $d$ is in position $(0,j)$ of $A$  (really $d=c$), by Lemma~\ref{odd-hook-notations}.
Hence we have that
$A( 2^\ell, \leftarrow d)^{\uparrow}=B(2^\ell, \leftarrow c)$, and therefore we have that
$f_\ell(C_{2^t}(\lambda))=C_{2^t}(f_\ell(\lambda))$, as desired.
On the other hand, if $x=j$ then $d$ is in position $(1, j)$ and again we deduce that $A(2^\ell, \leftarrow d)^{\uparrow}=B(2^\ell, \leftarrow c)$.
A similar argument is used for the case where $x=j-2^\ell$.
\end{proof}

\begin{lemma}
Let $n=2^t+m$ with $2\leq m<2^t$.
Then $\la\in \cG_0(n)$ implies $C_{2^t}(\lambda) \in \cG_0(m)$.
\end{lemma}
\begin{proof}
This is a consequence of Lemma~\ref{lem: f commutes with core}.
\end{proof}

For the notation in the following results we refer to Definition~\ref{def: leftslide} and Lemma~\ref{odd-hook-notations}.

\begin{lemma}\label{lem: goodcores}
Let $2\leq m<2^t$ and let $\gamma\vdasho m$.
Let $B$ be a normalized $2^t$-abacus for $\gamma$.
Then
$\gamma \in \cG_0(m)$
if and only if $r_0(B)=r_1(B)$ or $r_0(B)=r_1(B)-1$.
\end{lemma}
\begin{proof}
If $r_0(B)=r_1(B)=:r$ then there is a bead in position $(0,r)$ of $B$ (call it $c$) and there are gaps in positions $(0, r-2)$ and $(0, r-1)$ of $B$. Moreover, by definition we have that $B^0:=B(1, \leftarrow c)$ and $B^1:=B(2, \leftarrow c)$ are $2^t$-abaci for $f_0(\gamma)$ and $f_1(\gamma)$ respectively.
Clearly $B^1=B^0(1, \leftarrow c)$. Since $B^1$ is a $2^t$-abacus for an odd partition we deduce from Theorem~\ref{theo: INOT} that this partition is $f_0 f_0(\gamma)$. Hence $f_0 f_0(\gamma)=f_1(\gamma)$.

If $r:=r_0(B)=r_1(B)-1$ then we have beads $c$ and $d$ in positions $(0, r)$ and $(0, r+1)$ of $B$ respectively, and there is a gap in position $(0, r-1)$. By construction $$B^1:=B(2,\leftarrow d)=B^0(1, \leftarrow d'), $$
where $B^0:=B(1, \leftarrow c)$ and $d'$ is the bead in position $(0, r+1)$ of $B^0$.
Arguing as in the previous case we deduce that $f_0 f_0(\gamma)=f_1(\gamma)$.

In order to prove the converse let us suppose that $f_0f_0(\gamma)=f_1(\gamma)$ and let $r:=r_1(B)\in\{2, \ldots, 2^t-1\}$ be such that $r_0(B)\notin\{r, r-1\}$. Let $c$ be the bead in position $(0, r)$ of $B$, and observe that position $(0, r-2)$ of $B$ is empty. Hence $B^1:=B(2, \leftarrow c)$ is the $2^t$-abacus configuration for $f_1(\gamma)$ and has a bead in position $(0, r-2)$.

Suppose that position $(0, r-1)$ of $B$ is empty.
Since $r_0(B)\notin\{r, r-1\}$ we deduce that positions $(0, r-2)$ and $(0, r-1)$ are empty in $B^0=B(1, \leftarrow d)$, where $d$ is the bead in position $(0, r_0(B))$ and $B^0$ is the $2^t$-abacus for $f_0(\gamma)$. It follows that the abacus configuration $B^{0,0}$ for $f_0f_0(\gamma)$ obtained from $B^0$ by sliding one bead to its left has a gap in position $(0, r-2)$ and therefore we have $B^{0,0}\neq B^1$ and hence $f_0 f_0(\gamma)\neq f_1(\gamma)$.

We still have to consider the case where there is a bead in position $(0, r-1)$ of $B$. Again we deduce that $f_0f_0(\gamma)\neq f_1(\gamma)$ by arguing in a completely similar fashion as done above.
\end{proof}

For an odd partition $\gamma$, we denote by $\mathcal{E}(\gamma, 2^t)$ the set of all odd partitions that are obtained from $\gamma$ by adjoining a $2^t$-hook to $\gamma$. Recall that if $|\gamma|<2^t$ then $|\mathcal{E}(\gamma, 2^t)|=2^t$, by the main result of \cite{B}.

\begin{proposition}\label{prop: bases}
Let $n=2^t+m$ with $2\leq m< 2^t$.
Let $\gamma\in \cG_0(m)$
and let $\lambda\in\mathcal{E}(\gamma, 2^t)$.
Let $B$ be a normalized $2^t$-abacus for $\gamma$, and let $A$ be the $2^t$-abacus for $\lambda$ such that $A^\uparrow=B$. Let $x\in\{0, \ldots, 2^t-1\}$ be such that $w_{2^t}(A_x)=1$.
\begin{itemize}
\item[(i)] If $r_0(B)=r_1(B)=:r$, then
$\la\in \cG_0(n)$
if and only if $x\notin\{r-2, r-1\}$.
\item[(ii)] If $r_0(B)=r_1(B)-1=:r$, then
$\la\in \cG_0(n)$
if and only if $x\notin\{r, r+1\}$.
\end{itemize}
\end{proposition}
\begin{proof}
Suppose we are in case (i) where $r_0(B)=r_1(B)=:r$, and let us assume that $x\notin\{r-2,r-1,r\}$.
Let $c$ be the bead in position $(0,r)$ of $A$. Using Lemma~\ref{lem:maxhook}(3) we get that $A^0:=A(1, \leftarrow c)$ and $A^1:=A(2, \leftarrow c)$ are $2^t$-abacus configurations for $f_0(\lambda)$ and $f_1(\lambda)$, respectively.
Let $c'$ be the bead in position $(0,r-1)$ of $A^0$.
It follows that $A^{0,0}:=A^0(1,\leftarrow c')$ is a $2^t$-abacus configuration for $f_0f_0(\lambda)$.
Since $A^{0,0}=A^1$ we deduce that $f_0f_0(\lambda)=f_1(\lambda)$.
If $x=r$, we let $c$ be the bead in position $(1,r)$ of $A$.
Arguing exactly as above, we deduce that $f_0f_0(\lambda)=f_1(\lambda)$, also in this case.

Suppose now that $x=r-1$. Let $c$ be the bead in position $(0,r)$ of $A$ and let $d$ be the bead in position $(-1,r)$ of $A$.
Using Lemma~\ref{lem:maxhook}(3), we observe that $A^1:=A(2,\leftarrow c)$ is a $2^t$-abacus configuration for $f_1(\lambda)$.
On the  other hand we notice that $A^0:=A(1,\leftarrow d)$ is a $2^t$-abacus configuration for $f_0(\lambda)$.
It follows that $A^{0,0}:=A^0(1,\leftarrow e)$ is a $2^t$-abacus configuration for $f_0f_0(\lambda)$, where we denoted by $e$ the bead in position $(0,r-1)$ of $A^0$. Since $A^{0,0}\neq A^1$ are two abacus configurations with the same number of beads, we deduce that $f_0f_0(\lambda)\neq f_1(\lambda)$, by Lemma~\ref{abacus-facts}. With very similar ideas we deal with the case where $x=r-2$.

\smallskip

In order to prove the statement in case (ii) we use a very similar argument to the one used above to deal with case (i).
\end{proof}

For $n\in \N$ of the form $n=2^t$ or $n=2^t+1$,
the numbers $G_0(n)$ may easily be computed,
and these numbers will
provide crucial input in the general case.

\begin{proposition}\label{prop: basecases}

We have $G_0(1)=1$, $G_0(2)=2=G_0(3)$.

For $t>1$, the following holds.
\begin{itemize}
\item[{(a)}]
Let $\lambda=(2^t-u,1^u)\vdasho 2^t$,
for some $u\in \{0,1,\ldots,2^t-1\}$.
Then
$\la\in \cG_0(2^t)$
if and only if $u \equiv 0$ or $u \equiv 3 \mod 4$.
In particular, $G_0(2^t)=2^{t-1}$.
\item[{(b)}]
Let $\lambda\vdasho 2^t+1$.
Then
$\la\in \cG_0(2^t+1)$
if and only if $\lambda$ or $\lambda'$ is $(2^t+1)$ or $(2^t-1,2)$.
In particular, $G_0(2^t+1)=4$.
\end{itemize}
\end{proposition}
\begin{proof}
We may assume that $t>1$.

(a)
Using Lemma~\ref{lem: hooks} we compute for $\lambda=(2^t-u,1^u)$:
$$
f_0f_0(\lambda)=\left\{
\begin{array}{ll}
(2^t-u-2,1^u) & \text{for } u\equiv 0 \mod 4\\
(2^t-u-1,1^{u-1}) & \text{for } u\equiv 1 \text{ or } 2 \mod 4\\
(2^t-u,1^{u-2}) & \text{for } u\equiv 3 \mod 4\\
\end{array}
\right.
$$
and
$$
f_1(\lambda)=\left\{
\begin{array}{ll}
(2^t-u-2,1^u) & \text{for } u\equiv 0 \text{ or } 1 \mod 4\\
(2^t-u,1^{u-2}) & \text{for } u\equiv 2 \text{ or } 3\mod 4\\
\end{array}
\right.
$$
This implies the claim.
\smallskip

(b)
Note that $f_0(\lambda)$ is a hook. In the stated four cases, it is easy to see that
$\la\in \cG_0(2^t+1)$.
When we are not in one of these cases,
we have $\lambda=(2^t-u,2,1^{u-1})$ for some $u\in \{ 2,\ldots, 2^t-3\}$,
and $f_1(\lambda)$ is not a hook, so cannot coincide with $f_0f_0(\lambda)$.
\end{proof}

Thus the following gives an explicit formula for $G_0(n)$
which then allows to compute all values $G_k(n)$.

\begin{corollary}\label{cor:G0-count}
Let $n=2^{a_1}+ 2^{a_2}+\ldots + 2^{a_r}+\varepsilon \ge 2$,
where $a_1>\cdots >a_r>0$ and $\varepsilon\in\{0,1\}$. Then
$$G_0(n)=
G_0(2^{a_r}+\varepsilon)\cdot\prod_{j=1}^{r-1}(2^{a_j}-2).$$

In particular, $\cG_0(n)=\cO(n)$ only holds for $n\in \{2,3,5\}$,
and $\cG_0(n)=\{(n),(1^n)\}$ only holds for $n\in \{2,3,4\}$.
\end{corollary}

\begin{corollary}\label{cor:Gk-all}
Let $n\in \N$, $k>0$ with $2^{k+1}\le n$. Then $\cG_k(n)=\cO(n)$ if and only if  $n=2^{k+1}+m$, with $m<2^k$, i.e., if and only if \, $\lfloor n/2^k \rfloor=2$.
\end{corollary}

\begin{proof}
Clearly, $\cG_k(n)=\cO(n)$ if and only if $\cG_k(n;\tau)=\cO(n;\tau)$ for all $k$-types $\tau$ of size $w=\lfloor n/2^k \rfloor$.
With the notation in~\ref{not:33}, we get
$b_p=b_q$ by Lemma~\ref{lem: reduced-ff-count},
as otherwise there is a $k$-type $\tau$ with $\cG_k(n;\tau)=\emptyset$.
This implies $b_q>0$ and hence $w$ is even. Taking a $k$-type with $w$ as its only non-zero part, Lemma~\ref{lem: reduced-ff-count} and Corollary~\ref{cor:G0-count} then imply $w=2$.
\end{proof}

We also note the following implication of Corollary~\ref{cor: fk2-fk+1}
where the bound for $w$ is sharp by Corollary~\ref{cor:G0-count}:

\begin{corollary}\label{cor:Gk-small}
Assume $w=\lfloor n/2^k \rfloor\in \{2,3,4\}$.
Let $\lambda\vdasho n$ with
$\Qd{k}{\la}=(\la_1,\ldots,\la_{2^k})$.
Then $\lambda\in \cG_k(n)$ if and only if
there is an $s\in \{1,\ldots,2^k\}$ such that
$\la_s\in \{(w),(1^w)\}$.
\end{corollary}

\begin{example}
We illustrate the efficiency of the formulae for $G_k(n)$ obtained in this section
by considering
an example.

Let $n=24=2^4+2^3$ and let us consider $k=2$.
Referring to Notation~\ref{not:33},
we have $p=q=r=2$.
Hence, Theorem~\ref{theo: G} implies
$$G_{2}(24)=2^2 (w_G(\{2\})+w_G(\{1,2\}))
= 2^2 (2^2\cdot 3 \cdot G_0(2)+G_0(2+2^2)).$$
Now we use Corollary~\ref{cor:G0-count} and Proposition~\ref{prop: basecases} to
obtain the necessary values of $G_0$;
here, we have $G_0(2)=2$ and $G_0(6)=G_0(2)(2^2-2)=4$.
Thus we conclude $G_{2}(24)=112$.

It is illustrative to also go through the arguments again
that tell us how to detect whether a certain $\lambda\in\mathcal{O}(24)$ is such that $f_2f_2(\lambda)=f_3(\lambda)$.
By Lemma~\ref{lem: fk2-fk+1}, we just need to analyze the
$2$-data of $\la$. In particular, in this case we just need to look at $\Qd{2}{\la}=(\lambda_1,\lambda_2,\lambda_3,\lambda_4)$.
Referring to Notation~\ref{not:33}, we observe that $b_p=b_q=3-2=1$ in this case.
Therefore we can either have two non-empty odd partitions $\lambda_i, \lambda_j$ in
$\Qd{2}{\la}$ of size $2$ and $4$, respectively, such that $\lambda_i\in \mathcal{G}_0(2)$ (this gives us $(2^2\cdot G_0(2))\cdot (3\cdot 2^2)$ possibilities for $\lambda$) or we can have a unique non-empty partition $\lambda_i$ in $\Qd{2}{\la}$ of size $6$, such that $\lambda_i\in \mathcal{G}_0(6)$  (this gives us $2^2\cdot G_0(6)$ possibilities for $\lambda$).
\end{example}

We conclude the section by proving Theorem~\ref{ThmA} (below), as promised in the introduction. It shows that
when $\lfloor n/2^k \rfloor\le 4$,
then for $\la \vdasho n$ the information on whether $f_kf_k(\lambda)=f_{k+1}(\lambda)$ holds is encoded in the character table of~$\fS_n$.

\begin{teo}\label{ThmA}
Let $k,n\in \N$ with $2^{k+1}\le n$ and $w=\lfloor n/2^k \rfloor\le 4$; set $m=n-w\cdot 2^k$.
Let $\rho\in\fS_n$ be a product of $w$ $2^k$-cycles,
$\gamma=\omega_m$ on the fixed points of $\rho$; set $g=\rho\gamma\in \fS_n$.
Let $\la\vdasho n$.
Then $f_kf_k(\lambda)=f_{k+1}(\lambda)$ if and only if $|\chi^\lambda(g)|=1$.
\end{teo}
\begin{proof}
Let $\Qd{k}{\la} =(\lambda_1,\ldots, \lambda_{2^k})$.
By Corollary~\ref{cor:value1},
$|\chi^{\la}(g)|=1$ if and only if  $\lambda_s\in \{(w),(1^w)\}$
for some $s \in \{1,2,\ldots,2^k \}$.
As $w\le 4$, by Corollary~\ref{cor:Gk-small} this is equivalent to $\la\in \cG_k(n)$,
thus proving the claim.
\end{proof}

\begin{corollary}\label{cor:cortheoA}
Let $\la \vdash_o n$. 
Let $n=2^{a_1}+2^{a_2}+\ldots + 2^{a_r}$ where $a_1>a_2>\cdots >a_r$.
Choose  $k=a_i-1$ for some $i \in \{1,\cdots,r\}$.
Let $\omega_{n,i} \in\fS_n$ be an element of cycle type $(2^{a_1},2^{a_2}\cdots 2^{a_{i-1}},2^k,2^k, 2^{a_{i+1}},\cdots,2^{a_r})$.
If $f_kf_k(\la)=f_{k+1}(\la)$ then $|\chi^\la(\omega_{n,i}) |=1$.
\end{corollary}

\begin{proof}
Let $\la=\la^1,\ldots , \la^r$ be the 2-chain in $\la$.
Applying Lemma~\ref{lem:chain-value} $i-1$ times we get that
$|\chi^{\la}(\omega_{n,i}) |=|\chi^{\la^i}(\omega_{n,i}^*)|$
where $\omega_{n,i}^*$ is  an element of cycle type $(2^k,2^k, 2^{a_{i+1}},\cdots 2^{a_r})$.
We get from Theorem~\ref{ThmA} that $|\chi^\la(\omega_{n,i}) |=1$ if and only if $f_kf_k(\la^i)=f_{k+1}(\la^i)$.
By Lemma~\ref{lem: f commutes with core}
$f_kf_k(\la)=f_{k+1}(\la)$ implies $f_kf_k(\la^i)=f_{k+1}(\la^i)$.
\end{proof}

The following example shows that the converse of Corollary~\ref{cor:cortheoA} does not hold.

\begin{example}
Let $n=12=2^3+2^2$ and take $k=1$.
The cycle type of
$\omega_{12,2}$ is $(8,2,2)$.
Let $\la=(10,1^2) \vdash_o 12$.
We have that $|\chi^{\la}(\omega_{12,2}) |=1$.
However, $f_1f_1(\la)=(8)$ and $f_2(\la)=(6,1^2)$.
\end{example}

\section{Commutativity}\label{sec:commutativity}

Let $n,k,\ell$ be nonnegative integers such that
$k<\ell$ and $2^k+2^\ell\leq n$.
We want to investigate for which odd partitions $\lambda\in \cO(n)$
the commutativity relation
$f_kf_\ell(\lambda)=f_\ell f_k(\lambda)$ holds.
We define
\[
\T_{k,\l}(n) =
\{\la \vdash_o n \mid  f_kf_\l(\la) = f_\l f_k(\la) \}
\:,\:
\F_{k,\l}(n) =
\{\la \vdash_o n \mid  f_kf_\l(\la) \ne f_\l f_k(\la) \}
\]
and set $T_{k,\l}(n)=|\T_{k,\l}(n)|$, $F_{k,\l}(n)=|\F_{k,\l}(n)|$.
Furthermore, we set
$\T_{k,\l}= \bigcup_{n\ge 2^k+2^\l} \T_{k,\l}(n)$
and
$\F_{k,\l}= \bigcup_{n\ge 2^k+2^\l} \F_{k,\l}(n)$.

\smallskip

Write $n=2^t+m$ with $0\le m<2^t$.
By \cite[Theorem B]{BGO}
we know that
$\T_{k,\l}(n) = \cO(n)$
if and only if $m<2^k$, or $\ell=t$,
or $n=6$ and $(k,\l)=(0,1)$.
Therefore we may from now on assume that
$n>6$, and that
$\ell, k\in\mathbb{N}_0$ are such that $k<\ell<t$
and $2^k\le m$.

\smallskip

We fix below some notation that will be used
throughout this section.

\begin{notation}\label{not:44}
We write
$n=2^{a_1} + \ldots + 2^{a_r}$, where
$a_1>a_2 > \ldots > a_r$,
and define $p,q \in\{1,\ldots, r\}$
to be maximal such that
$a_p\ge \l$ and $a_q\ge k$, respectively.
Set $b_j=a_j-k$, for $1\le j \le q$, and
$B=\{b_1,\ldots,b_q\}$.

\end{notation}

Using a strategy similar to the one used in Section~\ref{sec: kk,k+1}, we start by observing that we can reduce the study of our problem to the case where the smaller  parameter $k$ is equal to $0$. This is shown in the following two lemmas.

\begin{lemma}\label{lem: k ell 0 ell}
Let $\lambda\vdasho n$ be such that $R_k^{k}(\lambda)\neq R_k^{\ell}(\lambda)$.
Then 
$\la \in \T_{k,\l}(n)$.
\end{lemma}
\begin{proof}
Let $\Qd{k}{\la} =(\lambda_1,\ldots, \lambda_{2^k})$.
Suppose that $R_k^k(\lambda)=\lambda_i\neq \lambda_j=R_k^\ell(\lambda)$, for some $i,j\in\{1,\ldots, 2^k\}$.
By Lemma~\ref{lem:R-arithmetic},
$2^{b_q}\subseteq_2 |\lambda_i|$ and
$2^{b_p} \subseteq_2 |\lambda_j|$.
Since $|f_0(\lambda_i)|=|\lambda_i|-1$ and $|f_{\ell-k}(\lambda_j)|=|\lambda_j|-2^{\ell-k}$,
we deduce that
$$\{|f_0(\lambda_i)|, |f_{\ell-k}(\lambda_j)|, |\lambda_s| \: \big| \:  s\in\{1,\ldots, 2^k\}\smallsetminus\{i,j\}\},$$ is a set consisting of pairwise $2$-disjoint natural numbers. Therefore we have that
$f_kf_\ell(\lambda)=f_\ell f_k(\lambda)$ is the partition with $k$-data obtained from the $k$-data of $\lambda$ by replacing $\lambda_i$ and $\lambda_j$
by $f_0(\lambda_i)$ and $f_{\ell-k}(\lambda_j)$,  respectively.
\end{proof}

\begin{lemma}\label{lem: k ell 0 ell 2}
Let $\lambda\vdasho n$ be such that
$R_k^{k}(\lambda)= R_k^{\ell}(\lambda)$. 
Then
$\la \in \T_{k,\l}(n)$
if and only if
$R_k^{k}(f_{\ell}(\lambda))= f_{\ell-k}(R_k^k(\la))$ and
$R_k^k(\la)\in \T_{0,\l-k}$.
\end{lemma}
\begin{proof}
Let $\Qd{k}{\la} =(\lambda_1, \lambda_2,\ldots, \lambda_{2^k})$ and $R_k^k(\la)=R_k^\l(\la)=\lambda_i$, for some $i\in\{1,\ldots, 2^k\}$.

Suppose first
that $\la \in \T_{k,\l}(n)$.
Comparing row $k$ of the $k$-data of $f_k(\lambda)$ with row $k$ of the $k$-data of $f_\ell(\lambda)$ we deduce that the two stated conditions must hold.

Conversely,
let $\rho$ be the partition of $n-2^k-2^{\ell}$ whose $k$-data is obtained from the $k$-data of $\lambda$ by replacing $\la_i$ with $f_0f_{\ell-k}(\la_i)$.
Since $R_k^{k}(f_{\ell}(\lambda))= f_{\ell-k}(R_k^k(\la))$,  we deduce that $\rho=f_kf_\ell(\lambda)$.
As $\la_i=R_k^k(\la)\in \T_{0,\l-k}$, the $k$-data of $\rho$ equals the $k$-data of $f_\ell f_k(\lambda)$. Hence we deduce $\rho=f_\ell f_k(\lambda)$, as required.
\end{proof}

With our arithmetical understanding of the
removal process (Lemma~\ref{lem:R-arithmetic}),
the condition
$R_k^{k}(f_{\ell}(\lambda))= f_{\ell-k}(R_k^k(\la))$
is not difficult to study; it is equivalent to say that $|R_k^k(\la)|>2^{\l-k}$ and $|f_0f_{\ell-k}(R_k^k(\la))|$ is $2$-disjoint from the sizes of the other $2^k-1$ partitions in $\Qd{k}{\la} $.

\smallskip

Lemmas~\ref{lem: k ell 0 ell} and~\ref{lem: k ell 0 ell 2} provide a reduction to the case where $k=0$. We will first discuss the reduction to this situation in detail, and then investigate the situation when $f_0f_\ell(\lambda)=f_\ell f_0(\lambda)$, for any given $\lambda\vdasho n=2^t+m$ where $1\leq m$ and $\ell<t$ (and $n>6$).

\begin{definition}
Let $n,k,\l \in \N$, $k<\l$, and $2^k+2^\l\le n$.
For a given $k$-type $\tau$ of size $\lfloor n/2^k \rfloor$, we
let $\T_{k,\l}(n;\tau)$,  $\F_{k,\l}(n;\tau)$
denote the subsets of $\T_{k,\l}(n)$, $\F_{k,\l}(n)$,
respectively,
of partitions of $k$-type $\tau$.
By $T_{k,\l}(n;\tau)$,  $F_{k,\l}(n;\tau)$
we denote the cardinalities of the corresponding sets.
\end{definition}

Using Lemmas~\ref{lem:R-arithmetic},~\ref{lem: k ell 0 ell}
and~\ref{lem: k ell 0 ell 2},
we easily deduce the following description
that reduces the problem for parameters $(k,\l)$
to that for $(0,\l-k)$:

\begin{lemma}\label{lem: reducedcount}
We let $\tau=(n_1,\ldots,n_{2^k})$ be a $k$-type, of size $\lfloor n/2^k\rfloor$.
\begin{enumerate}
\item
If $2^{b_p},2^{b_q}$ are binary digits of two different parts of $\tau$,
then $\T_{k,\l}(n,\tau)=\cO(n,\tau)$.
\item
Assume that $2^{b_p},2^{b_q}$ are binary digits of the same part of $\tau$, say $n_j$.\\
If $b_p=b_q>\l-k$, or if $b_p>b_q$, then
$$
\T_{k,\l}(n;\tau)=
\{\la \in \cO(n;\tau) \mid
Q_{2^k}(\la)_j \in \T_{0,\l-k}(n_j)
\}.
$$
\item
Assume $b_p=b_q=\l-k$. \\
If $2^{b_{p-1}},2^{b_p}$ are binary digits of the same part of $\tau$, say $n_j$,
then
$$
\T_{k,\l}(n;\tau)=
\{\la \in \cO(n;\tau) \mid
Q_{2^k}(\la)_j \in \T_{0,\l-k}(n_j)
\},
$$
otherwise
$\T_{k,\l}(n,\tau)=\emptyset$.
\end{enumerate}
\end{lemma}

Together with \cite[Theorem B]{BGO} this result tells us when for all odd partition of a fixed $k$-type commutativity holds.

\begin{corollary}
We let $\tau=(n_1,\ldots,n_{2^k})$ be a $k$-type, of size $\lfloor n/2^k\rfloor$.
Then $\T_{k;\l}(n,\tau)=\cO(n;\tau)$ if and only if
one of the following holds:
\begin{enumerate}
\item
$2^{b_p},2^{b_q}$ are binary digits of two different parts of $\tau$.
\item
$2^{b_p},2^{b_q}$ are binary digits of the same part $n_j>2^{\l -k}$ of $\tau$,
and
$2^{\l -k}=2^{b_p}$ is the largest binary digit of $n_j$,
or $n_j=2^{b_p}$, 
or $n_j=6$ and $(b_p,b_q,\l-k)\in \{(2,2,1),(2,1,1),(1,1,1)\}$.
\end{enumerate}
\end{corollary}

\begin{proof}
By \cite[Theorem B]{BGO} we know that
$\T_{0,\l-k}(n_j)=\cO(n_j)$
if and only if $2^{\l -k}$ is the largest binary digit
of $n_j$, or $n_j$ is a 2-power, or $n_j=6$ and $\l -k=1$.
The description of $\T_{k,\l}(n;\tau)$
in Lemma~\ref{lem: reducedcount}
then implies the stated assertion.
\end{proof}

From Lemmas~\ref{lem: k ell 0 ell} and~\ref{lem: k ell 0 ell 2} we also deduce that for $\la\in \cO(n)$,
the commutativity property does not depend on the rows of the $k$-data corresponding to the cores.
Hence we have:

\begin{lemma}\label{lem: core-part}
We let $\bar n= \sum_{i=1}^q 2^{a_i}$.
For any $k$-type $\tau$ of size $\lfloor n/2^k\rfloor$ we have
$$
T_{k,\l}(n;\tau)=(\prod_{j=q+1}^r 2^{a_j})\cdot T_{k,\l}(\bar n;\tau)
$$
In particular,
$$
T_{k,\l}(n)=(\prod_{j=q+1}^r 2^{a_j})\cdot T_{k,\l}(\bar n)\:.
$$
Analogous statements hold for $F_{k,\l}(n;\tau), F_{k,\l}(n)$.
\end{lemma}

Similarly as we have done it in
Section~\ref{sec: kk,k+1},
we introduce suitable weights that can be computed
from the numbers for the $(0,\l-k)$-case.

\begin{definition}
For $J\subseteq I=\{1,\ldots,q\}$
we define its $F$-weight with respect to the set
$B=\{b_1,\ldots,b_q\}$ to be
$$
w_F(J) =
(\prod_{i\in I\setminus J}2^{b_i})
\cdot
(2^k-1)^{|I\setminus J|}
\cdot
F_{0,\l -k}(\sum_{j\in J}2^{b_j}).
$$
\end{definition}

Using in part similar arguments as before,
we have the following formulae for the numbers $F_{k,\l}(n)$, based on the values $F_{0,\l-k}(u)$
(and hence also such formulae for $T_{k,\l}(n)$).

\begin{teo}\label{theo: B}
With the notation introduced in~\ref{not:44},
the following holds.

If $a_p>a_q$ or if $a_p=a_q>\l$, then
$$
F_{k,\l}(n)=
(\prod_{j=q+1}^r 2^{a_j})\cdot
2^k \cdot
\sum_{\{p,q\}\subseteq J \subseteq I} w_F(J) \:.
$$

If $a_p=a_q=\l$, then
$$
F_{k,\l}(n)=
(\prod_{j=q+1}^r 2^{a_j})\cdot
\biggl(
2^k \cdot
\bigl(
\sum_{\{p-1,p\}\subseteq J \subseteq I} w_F(J)
\bigr)
+2^{k(q-1)} (2^k-1) \prod_{i\in I}2^{b_i}
\biggr)
\:.
$$
\end{teo}

\begin{proof}
In view of Lemma~\ref{lem: core-part}, we focus on
computing the numbers $F_{k,\l}(\bar n)$.
We take the sum over all $k$-types of size $\lfloor n/2^k\rfloor$,
and for each $k$-type $\tau$
we count the partitions in $\F_{k,\l}(\bar n;\tau)$
according to Lemma~\ref{lem: reducedcount}.

When $a_p>a_q$, we only get a contribution from
those $k$-types $\tau$  where $2^{b_p},2^{b_q}$
are binary digits in the same part of $\tau$.
To construct these,  we take any $J\subseteq I$
containing $p,q$ and define $n_j:=\sum_{j\in J}2^{b_j}$.
We now consider the contributions from all $k$-types for this fixed~$J$.
We first have $2^k$ choices
for the position $j$ of the part
of which $2^{b_p},2^{b_q}$ are
binary digits.
For each $i\in I\setminus J$
we have $2^k-1$ choices to place $2^{b_i}$
as a binary digit of a part of the $\tau$ under construction, different from the $j$-th.
By Lemma~\ref{lem: reducedcount}, for such a $\tau=(n_1,\ldots,n_{2^k})$,
the partitions $\la\in \F_{k,\l}(\bar n;\tau)$
are then those
such that $\cQ(\la)_i$ is an arbitrary
odd partition of $n_i$ for $i\ne j$, and $\cQ(\la)_j\in \F_{0,\l-k}(n_j)$ (and the earlier rows of the $k$-data contain only empty cores). Hence we get
$|\F_{k,\l}(\bar n;\tau)|=(\prod_{i\in I\setminus J}2^{b_i})
\cdot
F_{0,\l -k}(\sum_{j\in J}2^{b_j})$, and the total contribution for $J$ is exactly $2^k w_F(J)$.

When $a_p=a_q$, the construction of the relevant $k$-types
$\tau$ is verbatim the same as before, noticing that now $\{p,q\}=\{p\}$.

In the situation where $a_p=a_q>\l$, for each $k$-type $\tau$ thus constructed the number
$|\F_{k,\l}(\bar n;\tau)|$ is determined as before,
giving again the total contribution $2^k w_F(J)$ for
each $J\subseteq I$ with $p\in J$.

In the remaining situation, where $a_p=a_q=\l$, we have to distinguish between the $k$-types $\tau$ where $2^{b_p}$ and $2^{b_{p-1}}$ are binary digits of the same part
of $\tau$, and those where this does not hold.
In the first case, take any $J\subseteq I$
containing $\{p-1,p\}$,
define $n_j:=\sum_{j\in J}2^{b_j}$
and argue as before to construct the $k$-types to~$J$.
Applying Lemma~\ref{lem: reducedcount},
the total contribution for $J$ is again
exactly $2^k w_F(J)$.
To construct the $k$-types in the second case,
note that having placed $2^{b_p}$ (in $2^k$ ways)
there are only $2^k-1$ choices to place
$2^{b_{p-1}}$ as a binary digit,
but for each of the other $q-2$
binary digits $2^{b_i}$ there are $2^k$ choices.
For any $\tau$ constructed in this way,
we know by Lemma~\ref{lem: reducedcount} that
$\F_{k,\l}(\bar n;\tau)=\cO(\bar n;\tau)$,
hence $|\F_{k,\l}(\bar n;\tau)|=\prod_{i\in I}2^{b_i}$.
Thus the total contribution over
all $k$-types of the second
form is exactly $2^{k(q-1)} (2^k-1) \prod_{i\in I}2^{b_i}$.
This proves the stated formula.
\end{proof}

\begin{examples}
We illustrate the efficiency of the formula above by giving some examples. Here we assume that we already know the values $F_{0,\l-k}(u)$; these -- or their counterparts $T_{0,\l-k}(u)$, respectively -- are computed explicitly in the next section.

(1)
Let $n=24=2^4+2^3$, $k=2$, $\l=3$.
Then $F_{2,3}(24)=2^2\cdot 2^3\cdot 3 + F_{0,1}(2^2+2^1)$.
As $F_{0,1}(6)=0$,
$F_{2,3}(24)=96$.

(2)
Let $n=36=2^5+2^2$, $k=2$, $\l=3$.
Then $F_{2,3}(36)=2^2 F_{0,1}(2^3+2^0)$.
As $F_{0,1}(9)=6$, $F_{2,3}(36)=24$.
\end{examples}

\section{The base case $\T_{0,\ell -k}(x)$}\label{sec:basecases}

We now turn to the discussion of
the situation when $f_0f_\ell(\lambda)=f_\ell f_0(\lambda)$, for any given $\lambda\vdasho n=2^t+m$ where $1\leq m<2^t$ and $\ell<t$.

\begin{definition}
For $n\in \N$ we denote by $\Omega(\ell; n)=\T_{0,\l}(n)$ the set consisting of all odd partitions $\lambda$ of $n$ such that $f_0f_\ell(\lambda)=f_\ell f_0(\lambda)$; for $n\le 2^\ell$, the set $\Omega(\ell; n)$
then is supposed to be the set of all odd partitions.
When $n=2^t+m$ as above,
and $\gamma\vdasho  m$, we set
$$\Omega(\ell; n\ |\ \gamma):=\{\lambda\in\Omega(\ell; n)\ |\ C_{2^t}(\lambda)=\gamma\}.$$
\end{definition}

\begin{strategy}
For any natural number $n\geq 2^{\ell+1}$ we write $n=2^{t_1}+\cdots +2^{t_a}+m$ where $t_1>t_2>\cdots >t_a>\ell$ and $m\leq 2^{\ell}$. We will first describe the set $\Omega(\ell; n)$ in terms of $\Omega(\ell; n-2^{t_1})$.

Afterwards, we will study the base cases, namely the set $\Omega(\ell; 2^t+m)$ for any $0\leq m\leq 2^{\ell}$ and $t\geq \ell$. To do this we will first analyse the case where $m\neq 2^\ell$. Separately we will deal with the remaining case: namely $m=2^\ell$ and $t>\ell$.

This will provide us with a complete recursive description of the set $\Omega(\ell; n)$. As an application of this, we derive a closed formula for $|\Omega(\ell; n)|$.
\end{strategy}

\begin{lemma}\label{lem: core}
Let $n=2^t+m$ where $ 2^{\ell}+1\leq m< 2^t$, and let $\lambda\vdasho n$. If $\lambda\in\Omega(\ell; n)$ then $C_{2^t}(\lambda)\in\Omega(\ell; m)$.
Thus, $$\Omega(\ell; n) = \bigcup_{\gamma\in\Omega(\ell; m)}\Omega(\ell; n\ |\ \gamma).$$
\end{lemma}
\begin{proof}
This follows directly from Lemma~\ref{lem: f commutes with core}.
\end{proof}

\subsection{Inductive step}

In light of Lemma~\ref{lem: core} we will now fix $\gamma\in\Omega(\ell; m)$ and explicitly describe which and how many of the $2^t$ odd extensions of $\gamma$ lie in $\Omega(\ell; n\ |\ \gamma)$.
In the following we are referring to Definition~\ref{def: odd-position} and
we are using Lemma~\ref{odd-hook-notations}.

\begin{definition}\label{def:types}
Let $n$ and $\ell$ be natural numbers such that $2^\ell\leq n$.
Let $\lambda\vdasho n$.
We say that $\lambda$ is of type $(1+)$
if $z_0(\la)=z_{\ell}(\la)$, and  $\lambda$ is of type $(1-)$
if $s_0(\la)=s_{\ell}(\la)$
(i.e., the odd 1-hook and odd $2^\ell$-hook of $\la$
have the same hand or foot, respectively).
Otherwise, we say that $\lambda$ is of type $(2)$.\\
When $\lambda\in\Omega(\ell; n)$,
we write $\lambda\in\Omega^{\epsilon}(\ell; n)$ if
$\la$ is of type $\epsilon\in \{1+,1-,2\}$.
\\
Clearly,
$\Omega(\ell; n)=\Omega^{1+}(\ell; n)\cup \Omega^{1-}(\ell; n)\cup \Omega^{2}(\ell; n)$. We set $\Omega^{1}(\ell; n):=\Omega^{1+}(\ell; n)\cup \Omega^{1-}(\ell; n)$.
\end{definition}

\begin{notation}\label{def: 2 to the t abacus}
Let $n$ and  $\lambda\vdasho n$ as above.
Assume that $n< 2^s$;
let $A$ be a normalized $2^s$-abacus for $\lambda$,  i.e.,
the beads after the first gap on $A$ correspond to the first column hook lengths in $\la$.
Let $c=c_0(A)$ and $d=c_\ell(A)$.
By Lemma~\ref{odd-hook-notations}, $c$ and $d$ are in positions $(0,r_0(\lambda))$ and $(0,r_\ell(\lambda))$ of $A$, respectively.

Note that $\la$ is of type $(1+)$ if $r_0(\la)=r_\ell(\la)$,
i.e., we have the situation of the \textit{same start bead},
while $\la$ is of type $(1-)$ if
$r_\ell(\la)= r_0(\la)+2^\ell-1$, i.e., the beads have the \textit{same target position}.

\end{notation}

For an illustration of this definition as well as for providing a basic case,
we consider the situation where $n$ is a 2-power. Note that in this case
$\Omega(\ell;n) = \mathcal{H}(n)$.

\begin{lemma}\label{lem: hook types}
Let $n=2^t$ and $0<\ell\le t$.
Let $\lambda\in \Omega(\ell;n)$, say $\lambda=(2^t-w,1^w)$ with $0\le w\le 2^t-1$.

For $\ell <t$,
$\lambda$ is of type $(1+)$ if
$1,2^\ell\not\subseteq_2 w$,
it is of type $(1-)$
$1,2^\ell\subseteq_2 w$,
and it is of type $(2)$ otherwise.
In particular,
$$|\Omega^{1\pm}(\ell; 2^t)|=2^{t-2}, \;
|\Omega^2(\ell; 2^t)|=2^{t-1}.$$

For $\ell =t$,
$\lambda$ is of type $(1+)$ if $w$ is even, and
it is of type $(1-)$ if $w$ is odd.
In particular,
$$|\Omega^{1\pm}(t, 2^t)|=2^{t-1}, \;
|\Omega^2(t, 2^t)|=0.$$
\end{lemma}
\begin{proof}
Let $r_0=r_0(\lambda)$,  $r_\ell=r_\ell(\lambda)$ be defined as above.
First we assume $\ell<t$.
By Lemma~\ref{lem: hooks} we have
$r_0=r_\ell$ if and only if $1,2^\ell\not\subseteq_2 w$,
and $r_\ell=r_0+2^\ell-1$ if and only if $1,2^\ell\subseteq_2 w$.
Similarly, in the case $\ell=t$, we have
$r_\ell=r_0$ if $w$ is even and
$r_\ell=r_0+2^\ell-1$ if $w$ is odd.
\end{proof}

\medskip

\begin{lemma}\label{lem: type1 extension and type preserved}
Let $n=2^t+m$ where $2^{\ell}+1\leq m< 2^t$, and let $\gamma\in\Omega^2(\ell; m)$. Then
$$\Omega(\ell; n\ |\ \gamma)=\mathcal{E}(\gamma, 2^t)\subseteq \Omega^2(\ell; n).$$
\end{lemma}
\begin{proof}
Let $\lambda\in\mathcal{E}(\gamma, 2^t)$. Let $B$ be a normalized $2^t$-abacus for $\gamma$ and let $A$ be the $2^t$-abacus for $\lambda$ such that $A^\uparrow=B$. Let $c:=c_0(B)$ and $d:=c_\ell(B)$. Thus $B(1, \leftarrow c)$ is a $2^t$-abacus for $f_0(\gamma)$ and $B(2^\ell, \leftarrow d)$ is a $2^t$-abacus for $f_\ell(\gamma)$. Following Notation~\ref{def: 2 to the t abacus} we have that $c$ is in position $(0,r_0(\gamma))$ and $d$ is in position $(0,r_\ell(\gamma))$.
Since $\gamma$ is of type $(2)$ we have that $r_\ell(\gamma)\notin\{r_0(\gamma), r_0(\gamma)+2^\ell-1\}$.
Moreover, since $f_0f_\ell(\gamma)=f_\ell f_0(\gamma)$ we deduce that $r_0(f_\ell(\gamma))=r_0(\gamma)$ and that $r_\ell(f_0(\gamma))=r_\ell(\gamma)$.
More precisely we have that $$(B(1,\leftarrow c))(2^\ell, \leftarrow d)=(B(2^\ell, \leftarrow d))(1,\leftarrow c),$$ is a $2^t$-abacus for $f_0f_\ell(\gamma)=f_\ell f_0(\gamma)$.

Let $x\in\{0,\ldots, 2^t-1\}$ be such that $w_{2^t}(A_x)=1$.
If $x\notin\{r_0(\gamma)-1, r_0(\gamma), r_\ell(\gamma)-2^\ell, r_\ell(\gamma)\}$ then $c_0(A)$ and $c_\ell(A)$ are in positions $(0, r_0(\gamma))$ and $(0, r_\ell(\gamma))$ of $A$ respectively,
by Lemma~\ref{lem:maxhook}(3).
Moreover, positions $(0, r_0(\gamma)-1)$ and $(0, r_\ell(\gamma)-2^{\ell})$ of $A$ are empty, because they were empty in $B$.
By Lemma~\ref{lem:maxhook}(3), we deduce that $$A^{\ell,0}:=[A(2^\ell, \leftarrow c_\ell(A))](1,\leftarrow c_0(A)),$$ is a $2^t$-abacus for $f_\ell f_0(\lambda)$. Similarly $A^{0,\ell}:=[A(1, \leftarrow c_0(A))](2^\ell,\leftarrow c_\ell(A))$ is a $2^t$-abacus for $f_0f_\ell(\lambda)$. Since $A^{0,\ell}$ is equal to $A^{\ell,0}$ we deduce that $\lambda\in\Omega(\ell; n\ |\ \gamma)$.

Suppose now that $x\in\{r_0(\gamma)-1, r_0(\gamma), r_\ell(\gamma)-2^\ell, r_\ell(\gamma)\}$.
For instance let $x=r_0(\gamma)$. Then $c_0(A)$ is the bead in position $(1, r_0(\gamma))$ of $A$, by Lemma~\ref{lem:maxhook}(3). On the other hand $c_\ell(A)=c_\ell(B)$ is the bead in position $(0, r_\ell(\gamma))$ of $A$. It is now clear that the $2^t$-abacus configuration obtained from $A$ by moving bead $c_0(A)$ to position $(1, r_0(\gamma)-1)$ and bead $c_\ell(A)$ to position $(0, r_\ell(\gamma)-2^\ell)$ is a $2^t$-abacus configuration for both $f_0 f_\ell(\lambda)$ and $f_\ell f_0(\lambda)$. Hence they coincide.
A very similar argument applies to the cases where $x\in\{r_0(\gamma)-1, r_\ell(\gamma)-2^\ell, r_\ell(\gamma)\}$.
We conclude that $\Omega(\ell; n\ |\ \gamma)=\mathcal{E}(\gamma, 2^t)$.

From the discussion above it follows immediately that given any $\lambda\in\Omega(\ell; n\ |\ \gamma)$ and considering  the $2^t$-abacus configuration $A$ for $\lambda$ as described above, we have that $c_0(A)$ lies on runner $A_{r_0(\gamma)}$ and $c_\ell(A)$ lies on runner $A_{r_\ell(\gamma)}$. Since $\gamma$ is of type $(2)$, this is enough to deduce that $\lambda$ is of type $(2)$ and therefore that
$\Omega(\ell; n\ |\ \gamma)\subseteq \Omega^2(\ell; n)$.
\end{proof}

\begin{lemma}\label{lem: 1+}
Let $n=2^t+m$ where $2^{\ell}+1\leq m< 2^t$, and let $\gamma\in\Omega^{1+}(\ell; m)$ and $\lambda\in\mathcal{E}(\gamma, 2^t)$. Let $B$ be a normalized $2^t$-abacus for $\gamma$ and $A$ be the $2^t$-abacus for $\lambda$ such that $A^\uparrow=B$.
Let $x\in\{0,1,\ldots, 2^t-1\}$ be such that $w_{2^t}(A_x)=1$. Then $\lambda\in\Omega(\ell; n\ |\ \gamma)$ if and only if $x\notin\{r_0(\gamma)-1, r_\ell(\gamma)-2^\ell, r_\ell(\gamma)-2^\ell-1\}$.
Moreover, $\Omega(\ell; n\ |\ \gamma)\subseteq \Omega^{1+}(\ell; n)$.
\end{lemma}

\begin{proof}
Let $\gamma$ be such that $r_0(\gamma)=r_\ell(\gamma)=:r$. Let $c:=c_0(B)=c_\ell(B)$. It is easy to see that the $2^t$-abacus $E$ obtained from $B$ by moving bead $c$ from position $(0, r)$ to position $(0, r-2^\ell-1)$ is a $2^t$-abacus for $f_0f_\ell(\gamma)=f_\ell f_0(\gamma)$.

Suppose that $x\notin\{r_0(\gamma)-1, r_\ell(\gamma)-2^\ell, r_\ell(\gamma)-2^\ell-1\}=\{r-1, r-2^\ell, r-2^\ell-1\}$.
If $x\ne r$, let $c'$ ($=c$)
be the bead in position $(0, r)$ of $A$,
for $x=r$, let $c'$ be the bead in position $(1,r)$ of $A$.
We deduce from Lemma~\ref{lem:maxhook}(3) that $A(1, \leftarrow c')$, $A(2^\ell, \leftarrow c')$ and $A(1+2^\ell, \leftarrow c')$ are $2^t$-abacus configurations for odd partitions $\mu_1, \mu_2$ and $\mu_3$ of $n-1, n-2^\ell$ and $n-2^\ell-1$ respectively. In particular $\mu_1=f_0(\lambda)$ and $\mu_2=f_\ell(\lambda)$. Moreover $\mu_3$ is an odd partition of $n-2^\ell-1$ obtained from $\mu_2$ by removing a $1$-hook. Hence we deduce that $\mu_3=f_0(\mu_2)=f_0f_\ell(\lambda)$. On the other hand we also notice that $\mu_3$ is obtained from $\mu_1$ by removing a $2^\ell$-hook. Hence $\mu_3=f_\ell(\mu_1)=f_\ell f_0(\lambda)$.
We conclude that $\lambda\in\Omega^{1+}(\ell; n)$.

Suppose now that $x\in\{r-2^\ell-1, r-2^\ell, r-1\}$. We will now show that the $2^t$-abaci for $f_0f_\ell(\lambda)$ and $f_\ell f_0(\lambda)$ obtained from $A$ by performing the appropriate bead moves are distinct. Suppose that $x=r-1$ (the other cases are treated in a completely similar way). Then $A$ is obtained from $B$ by sliding down the bead in position $(-1, r-1)$ to position $(0, r-1)$. Hence position $(-1, r-1)$ is empty in $A$. On the other hand, there are beads in positions $(-1, r)$ and  $(-1, r-1-2^\ell)$ of $A$, since $A$ coincides with $B$ on every runner distinct from runner $r-1$. Let $d$ be the bead in position $(-1, r)$ of $A$.
From Lemma~\ref{lem:maxhook}(3), we see that $A^0:=A(1, \leftarrow d)$ is a $2^t$-abacus for $f_0(\lambda)$. Let $A^{0,\ell}$ be the $2^t$-abacus configuration for $f_\ell f_0(\lambda)$ obtained from $A^0$ by moving a bead left. Since there is a bead in position $(-1, r-1-2^\ell)$ of $A^0$, we deduce that $A^{0, \ell}$ is not obtained by moving again bead~$d$;
in particular, $A^{0,\ell}$ has a bead in position $(-1, r-1)$.
On the other hand, calling $e$ ($=c$) the bead in position $(0, r)$ of $A$,  it follows from Lemma~\ref{lem:maxhook}(3) that $A^\ell:=A(2^\ell, \leftarrow e)$ is the $2^t$-abacus for $f_\ell(\lambda)$. Let $e'$ be the bead in position $(0, r-2^\ell)$ of $A^\ell$. Arguing as above one observes that $A^{\ell, 0}:=A(2^\ell+1, \leftarrow e)=A^\ell(1, \leftarrow e')$ is the $2^t$-abacus configuration for $f_0 f_\ell(\lambda)$. It is clear that position $(-1, r-1)$ of $A^{\ell, 0}$ is empty. Hence $A^{\ell, 0}\neq A^{0,\ell}$ and therefore
$\lambda\notin\Omega(\ell; n)$.
\end{proof}

\begin{lemma}\label{lem: 1-}
Let $n=2^t+m$ where $2^{\ell}+1\leq m< 2^t$, and let $\gamma\in\Omega^{1-}(\ell; m)$ and $\lambda\in\mathcal{E}(\gamma, 2^t)$. Let $B$ be a normalized $2^t$-abacus for $\gamma$ and $A$ be the $2^t$-abacus for $\lambda$ such that $A^\uparrow=B$.
Let $x\in\{0,1,\ldots, 2^t-1\}$ be such that $w_{2^t}(A_x)=1$. Then $\lambda\in\Omega(\ell; n\ |\ \gamma)$ if and only if $x\notin\{r_0(\gamma), r_\ell(\gamma), r_\ell(\gamma)+1\}$.
Moreover, $\Omega(\ell; n\ |\ \gamma)\subseteq \Omega^{1-}(\ell; n)$.
\end{lemma}
\begin{proof}
Let $\gamma$ be such that $r_\ell(\gamma)=r_0(\gamma)+2^\ell -1$ and
set $r:=r_0(\gamma)$. Let $c:=c_0(B)$.
Since $f_0f_\ell(\gamma)=f_\ell f_0(\gamma)$,
it is easy to see that $B$ has a bead $d$ at position $(0,r+2^\ell)$.
Then the $2^t$-abacus obtained from $B$ by moving bead $c$ from position $(0, r)$ to position $(0, r-1)$ and bead $d$ to position $(0,r)$ is a $2^t$-abacus for
$f_\ell f_0(\gamma)$.

The further arguments are entirely similar to the proof of
Lemma~\ref{lem: 1+}.
\end{proof}

\begin{corollary}\label{cor: iteration counting}
Let $n=2^t+m$ where $2^{\ell}+1\leq m< 2^t$. Then
$$|\Omega^{1\pm}(\ell; n)|=(2^t-3)|\Omega^{1\pm}(\ell; m)|
\; \text{ and } \;
|\Omega^{2}(\ell; n)|=2^t|\Omega^{2}(\ell; m)|.$$
In particular,
$$T_{0,\l}(n)=|\Omega(\ell; n)|=(2^t-3)|\Omega^{1}(\ell; m)|+2^t|\Omega^{2}(\ell; m)|.$$
\end{corollary}

\begin{proof}
Recalling that $|\mathcal{E}(\gamma, 2^t)|=2^t$ whenever $|\gamma|<2^t$ by the main result of \cite{B},
the statement follows directly from Lemmas~\ref{lem: type1 extension and type preserved},~\ref{lem: 1+} and~\ref{lem: 1-}.
\end{proof}

In order to obtain an explicit closed formula for $|\Omega(\ell; n)|$ for any given natural number~$n$, it is now enough to compute the values $|\Omega^{1}(\ell; m)|$ and $|\Omega^{2}(\ell; m)|$ for any $m\in\mathbb{N}$ such that $2^\ell +1 \leq m< 2^{\ell+1}$ or such that $m=2^t+q$ for some $0\leq q\leq 2^\ell$.

\subsection{Base cases}
In this subsection we let $n=2^t+m$ with $m<2^t$, where $t\geq \ell$ and $1\leq m\le 2^\ell$.
The following setup and notation for the case where $m<2^\ell$
will be used throughout most of this section, before we turn to the case $m=2^\ell$.

\begin{notation}\label{not: 3.10}
Let $\lambda\vdasho n$ and let $\gamma:=C_{2^t}(\lambda)\vdasho m$.
Let $B$ be a normalized $2^{\ell}$-abacus for $\gamma$.
Since $m<2^\ell$ we know from Lemma~\ref{odd-hook-notations} that $B$ does not have beads in rows $i\geq 1$. Let $c=c_0(B)$ so that $B^0:=B(1,\leftarrow c)$ is a $2^\ell$-abacus for $f_0(\gamma)$. Then $c$ is in position $(0,b_0)$ of $B$, where $b_0=r_0(B) \in\{1,\ldots, 2^\ell-1\}$.

Let $A$ be the $2^\ell$-abacus for $\lambda$ such that $A^\uparrow=B$. For $j\in\{0,\ldots, 2^{\ell}-1\}$ let $\mu_j$ be the partition represented on runner $A_j$ (seen as a $1$-abacus).
As observed in Remark~\ref{rem: abacus}, we know that $(\mu_0,\ldots, \mu_{2^{\ell}-1})$ equals $\mathcal{Q}_{\ell}(\lambda)$ (up to a cyclic shift of its components). Hence we deduce that there exists a unique $x\in\{0,\ldots, 2^\ell-1\}$ such that $\mu_x\in\mathcal{H}(2^{t-\ell})$ and such that $\mu_y=\emptyset$ for all $y\neq x$.
In particular we have that $w_{2^\ell}(A)=w_{2^\ell}(A_x)=2^{t-\ell}$ and $w_{2^\ell}(A_y)=0$ for all $y\neq x$. Let $A^\ell$ be the $2^\ell$-abacus for $f_\ell(\lambda)$ such that $(A^\ell)^\uparrow=B$.
The above discussion shows that $A^\ell$ is obtained from $A$ by sliding a bead up one row on runner $A_x$, to obtain a $1$-abacus configuration for $f_0(\mu_x)$. On the other hand a $2^\ell$-abacus for $f_0(\lambda)$ can be obtained from $A$ by sliding a bead $c'$ in position $(i, b_0)$ of $A$ to position $(i, b_0-1)$. This follows from Lemma~\ref{lem:maxhook}(3) because
$w_{2^\ell}(A(1,\leftarrow c'))=w_{2^\ell}(A)$ (by direct verification, or more generally by \cite[Lemma 2.5]{GLM}) and $(A(1,\leftarrow c'))^\uparrow=B^0$. We call $A^0:=A(1,\leftarrow c')$.
\end{notation}

\begin{lemma}\label{lem: runner x not b0}
Let $\lambda\vdasho n$ be as in Notation~\ref{not: 3.10}. If $x\notin\{b_0-1, b_0\}$ then $f_0f_\ell(\lambda)= f_\ell f_0(\lambda)$. In this case, $\lambda\in\Omega^2(\ell; n)$.
\end{lemma}
\begin{proof}
Let $A^\ell$ be the $2^\ell$-abacus for $f_\ell(\lambda)$ described in Notation~\ref{not: 3.10}.
Let $d$ be the bead in position $(0,b_0)$ of $A^\ell$. (There is such a bead in $A^\ell$ because $(A^\ell)_y=B_y$ for all $y\neq x$). We claim that $A^{\ell, 0}:=A^\ell(1, \leftarrow d)$ is a $2^\ell$-abacus for $f_0f_\ell(\lambda)$. By Theorem~\ref{theo: INOT}, we observe that in order to show this it is enough to show that $A^{\ell, 0}$ is a $2^\ell$-abacus for an odd partition of $n-2^\ell-1$. The previous statement holds by Theorem~\ref{thm:oddcriterion} because $w_{2^\ell}(A^{\ell,0})=w_{2^\ell}(A^\ell)$ and $(A^{\ell, 0})^\uparrow = B^0$.

Let $A^0$ be the $2^\ell$-abacus for $f_0(\lambda)$ described at the end of Notation~\ref{not: 3.10}.
Since $x\notin\{b_0-1, b_0\}$, we observe that $w_{2^\ell}((A^0)_y)=0$ for all $y\neq x$ and that $(A^0)_x=A_x$ (here we regard two runners equal if they coincide as $1$-abaci). In particular $(A^0)_x$ is a $1$-abacus for $\mu_x$ (as defined in Notation~\ref{not: 3.10}).
This shows that a $2^\ell$-abacus $A^{0,\ell}$ for $f_\ell f_0(\lambda)$ can be obtained from $A^0$ by sliding a bead up one row on runner $(A^0)_x$ to obtain a $1$-abacus for $f_0(\mu_x)$. Hence we deduce that $(A^{0,\ell})_{x}=(A^{\ell})_{x}=(A^{\ell,0})_{x}$. Similarly we also have that $(A^{0,\ell})_{y}=(A^{\ell,0})_{y}$ for all $y\neq x$.
\end{proof}

\begin{lemma}\label{lem: runner b0}
Let $\lambda\vdasho n$ be as in Notation~\ref{not: 3.10},
$u=2^{t-\ell}$. If $x=b_0$ then $f_0f_\ell(\lambda)= f_\ell f_0(\lambda)$ if and only if $\mu_{x}=(u)$. In this case, $\lambda\in\Omega^{1+}(\ell; n)$.
\end{lemma}
\begin{proof}
Suppose that $\mu_x=(u)$. Let $d$ be the bead in $A$ lying in position $(u, x)$. It is easy to see that both the abaci $A^{\ell, 0}$ and $A^{0,\ell}$ are obtained from $A$ by moving bead $d$ to position $(u-1, x-1)$. Hence they coincide and therefore we have that $f_0f_\ell(\lambda)= f_\ell f_0(\lambda)$.

Suppose now that $\mu_x=(u-w, 1^w)$ for some $w\in \{1,\ldots,u-1\}$.
Let us also assume that $w$ is odd (the case $w$ even is completely analogous).
Then we have beads in positions $(j, x)$ of $A$ for all $j\in\{-w+1, -w+2,\ldots, 0\}$ and a gap in position $(-w, x)$.
Moreover, the highest gap on runner $x-1$ of $A$ is in position $(0, x-1)$ since $A_{x-1}=B_{x-1}$.
This shows that $A^0=A(1,\leftarrow d)$ where $d$ is the bead in position $(0,x)$ of $A$ (this follows from Lemma~\ref{lem:maxhook}(3) and Theorem~\ref{theo: INOT}).
We observe that $(A^0)_x$ is a $1$-abacus configuration for the partition $\overline{\mu_x}:=(u-w+1, 1^{w-1})$.

Let $c'$ be the bead in position $(-w+1,x)$ of $A$ and let $e$ be the bead in position $(u-w,x)$ of $A$.

By Lemma~\ref{lem: hooks} we have that $f_0(\mu_x)=(u-w, 1^{w-1})$. Hence $A^\ell$ is obtained from $A$ by sliding bead $c'$ up to position $(-w,x)$. Since there is a bead in position $(-w,x-1)$ of $A$ we deduce that $A^{\ell, 0}$ has a bead in position $(-w,x)$.

On the other hand also $f_0(\overline{\mu_x})=(u-w, 1^{w-1})$, by Lemma~\ref{lem: hooks}.
Therefore $A^{0,\ell}$ is obtained from $A^0$ by sliding bead $e$ up to position $(u-w-1,x)$.
Hence $A^{0,\ell}$ has a gap in position $(-w,x)$. We conclude that $A^{0,\ell}\neq A^{\ell,0}$, as desired.
\end{proof}

\begin{lemma}\label{lem: runner b0-1}
Let $\lambda\vdasho n$ be as in Notation~\ref{not: 3.10}. If $x=b_0-1$ then $f_0f_\ell(\lambda)= f_\ell f_0(\lambda)$ if and only if $\mu_{x}=(1^{2^{t-\ell}})$. In this case, $\lambda\in\Omega^{1-}(\ell; n)$.
\end{lemma}
\begin{proof}
In this case we observe that $\lambda'$ satisfies the hypothesis of Lemma~\ref{lem: runner b0}. Hence the statement follows.
\end{proof}

\begin{teo}
Let $\lambda\vdasho n$;
we assume Notation~\ref{not: 3.10} and set
$u=2^{t-\ell}$. Then the following are equivalent.
\begin{itemize}
\item[(i)] $f_0f_\ell(\lambda)\neq f_\ell f_0(\lambda)$.

\item[(ii)] $x=b_0$ and $\mu_x\neq (u)$ or  $x=b_0-1$ and $\mu_x\neq (1^{u})$.

\item[(iii)] $\lambda$ or $\lambda'$ are such that $x=b_0$ and $\mu_x\neq (u)$.
\end{itemize}
\end{teo}

\begin{proof}
This is a consequence of Lemmas~\ref{lem: runner x not b0},~\ref{lem: runner b0},~\ref{lem: runner b0-1}.
\end{proof}

The unique base case still missing is the following for which we have to adapt our notation, using again the $2^t$-abacus as in Notation~\ref{def: 2 to the t abacus}.

\begin{notation}\label{not: m=2^l}
Let $n=2^t+2^{\ell}$, with $0<\ell < t$, i.e., we now consider the case $m=2^\ell$.
\\
Let $\lambda\vdasho n$,
and let $\gamma:=C_{2^t}(\lambda)\vdasho m$. Since $m=2^\ell$,
$\gamma$ is a hook, say $\gamma=(m-w,1^w)$, for some $w\in \{0,1,\ldots,m-1\}$.
Let $B$ be a normalized $2^t$-abacus for $\gamma$. Then $B$ does not have beads in rows $i\geq 1$, by Lemma~\ref{odd-hook-notations}.
Let $c=c_0(B)$ so that $B^0:=B(1,\leftarrow c)$ is a $2^t$-abacus for $f_0(\gamma)$.
Since $\gamma$ is a hook,
$c$ is in position $(0,b_0)$ of $B$ with $b_0=1$ or $b_0=2^\ell$.
The bead $d=c_\ell(B)$ in $B$ such that $B^\ell:=B(2^\ell,\leftarrow d)$
is a $2^t$-abacus for $f_\ell(\gamma)=\emptyset$
is clearly in position $(0,b_\ell)$ with $b_\ell=2^\ell$.

Let $A$ be the $2^t$-abacus for $\lambda$ such that $A^\uparrow=B$. For $j\in\{0,\ldots, 2^t-1\}$ let $\mu_j$ be the partition represented on runner $A_j$ (seen as a $1$-abacus).
So there exists a unique $x\in\{0,\ldots, 2^t-1\}$ such that $\mu_x=(1)$ and $\mu_y=\emptyset$ for all $y\neq x$; we also allow $x=-1$ and
identify $\mu_{-1}=\mu_{2^t-1}$ if needed.
As before, we note that a $2^t$-abacus for $f_0(\lambda)$ can be obtained from $A$ by moving a bead $c'$ one position to the left. We call $A^0:=A(1,\leftarrow c')$.
Similarly, let $A^\ell:=A(2^\ell,\leftarrow d')$
be the $2^t$-abacus for $f_\ell(\lambda)$,
obtained from $A$ by moving a bead $d'$ $2^\ell$ positions to the left.

Since $f_k(C_{2^t}(\lambda))=C_{2^t}(f_k(\lambda))$ for $k\le \ell$
by Lemma~\ref{lem: f commutes with core},
we know that $c'$ and $d'$ are on the same runners $b_0$ and
$b_\ell$, respectively,
as the beads $c$ and $d$ of $B$.
\end{notation}

\begin{lemma}\label{lem: m=2 to ell}
Let $n=2^t+2^{\ell}$, with $0<\ell < t$.
Let $\gamma\in \Omega(\ell;2^\ell)=\Omega^1(\ell;2^\ell)$, $\gamma=(2^\ell-w,1^w)$.
Let $\lambda\in \mathcal{E}(\gamma,2^t)$
and
$x$ as in Notation~\ref{not: m=2^l}.
Then
$\lambda\in \Omega(\ell;n\mid \gamma)$
if and only if
$x\in \{0,(-1)^{w-1},2^\ell,2^\ell+(-1)^{w-1}\}$.

More precisely, when $\gamma \in \Omega^{1+}(\ell; 2^\ell)$, then
$\lambda\in \Omega^{1+}(\ell;n\mid \gamma)$ for $x\in \{2^\ell-1,2^\ell\}$
and $\lambda\in \Omega^{2}(\ell;n\mid \gamma)$ for $x\in \{0,2^t-1\}$.
When $\gamma \in \Omega^{1-}(\ell; 2^\ell)$, then
$\lambda\in \Omega^{1-}(\ell;n\mid \gamma)$ for $x\in \{0,2^\ell+1\}$
and $\lambda\in \Omega^{2}(\ell;n\mid \gamma)$ for $x\in \{1,2^\ell\}$.
In particular,
$$
|\Omega^2(\ell;n)|=2^{\ell+1},\;
|\Omega^{1\pm}(\ell;n)|=2^{\ell}   
\; \text{ and } \;
T_{0,\l}(n)=|\Omega(\ell;n)|=2^{\ell+2}\:.$$
\end{lemma}
\begin{proof}
We consider the $2^t$-abacus of $\lambda$ using the notation introduced above.
Clearly, as $b_0\in \{1,2^\ell\}$ and $b_\ell=2^\ell$, the partition
$\gamma$ is always of type $(1)$; it is of type~$(1-)$, i.e.,
$b_0=1$, if and only if $w$ is odd, by Lemma~\ref{lem: hook types}.

We discuss the case $b_0=1$, i.e., $w$ odd, in detail; the case $b_0=b_\ell$
is similar (and just corresponds to the conjugate partition).

First we want to show that $f_\ell f_0(\lambda)=f_0f_\ell(\lambda)$
when $x\in \{0,1,2^\ell,2^\ell+1\}$.

For $x=0$, $\lambda=(m-w,1^{2^t+w})$, and we easily compute
$f_\ell f_0(\lambda)=(m-w,1^{2^t-m+w-1})=f_0f_\ell(\lambda)$;
we also note that $\lambda\in \Omega^{1-}(\ell;n)$.
For $x=2^\ell +1$, we use the $2^t$-abacus to compute again
$f_0f_\ell(\lambda)=(m-w,1^{2^t-m+w-1})$, which then also equals
$f_\ell f_0(\lambda)$. Again, $\lambda\in \Omega^{1-}(\ell;n)$.

For $x=1$, the bead $c'$ is at position $(1,1)$ in $A$,
$d'$ is at $(0,2^\ell)$, and both target positions $(1,0)$
and $(0,0)$ are empty; then we obtain the $2^t$-abacus
$A^{0,\ell}:=[A(1,\leftarrow c')](2^\ell,\leftarrow d')$
for $f_\ell f_0(\lambda)$ and
the $2^t$-abacus
$A^{\ell,0}:=[A(2^\ell,\leftarrow d')](1,c')$ for $f_0f_\ell(\lambda)$,
and these coincide.
The case $x=2^\ell$ is handled analogously.
In both cases, we easily see that $\lambda\in \Omega^2(\ell;n)$.

Note that in the case where $\gamma\in \Omega^{1+}(\ell;m)$,
i.e., $b_0=b_\ell$, a similar argument gives
two extensions $\lambda\in \Omega(\ell;n)$
of type $(1+)$ and two extensions of type $(2)$ as claimed.

Continuing with the case $b_0=1$, we next
want to show that $f_\ell f_0(\lambda)\ne f_0f_\ell(\lambda)$
when $x\not\in \{0,1,2^\ell,2^\ell+1\}$.

We have three types of extensions $\lambda$, corresponding to
$x\in \{2,\ldots,w\}$ if $w>1$, $x\in \{w+1,\ldots,2^\ell-1\}$
if $w<2^\ell -1$, and $x\in \{2^\ell+2,\ldots,2^t-1\}$ if $(\ell;t)\ne (1,2)$.

In all cases,
the $2^t$-abacus $A^\ell$ has beads at positions $(0,0)$ and $(0,1)$, so also
$A^{\ell,0}$ has this property. But $A^0$ has a gap at position $(0,1)$,
which can then not be filled in $A^{0,\ell}$
as $A^0$ has no bead in position $(0,2^\ell+1)$.
So $A^{\ell,0}\ne A^{0,\ell}$, and the claim follows.
\end{proof}

\medskip

We collect all the counting formulae
for $T_{0,\l}(n)$ in the following
result.
Together with the reduction formula for
the numbers $F_{k\l}(n)$ (so also for $T_{k,\l}(n)$)
in Theorem~\ref{theo: B} we thus have a complete answer to
our counting problem.

\begin{corollary} \label{cor: countfinal}
Let $n=2^{a_1}+\cdots+2^{a_r}$
where $a_1>a_2>\cdots >a_r$,
and let $0<\ell\le a_1$.
\begin{enumerate}
\item\label{cor: count1}
Assume $a_r >\ell$.
Then
$$T_{0,\l}(n)=
|\Omega(\ell;n)|=
2^{a_r-1}\big(\prod_{j=1}^{r-1}(2^{a_j}-3)
+\prod_{j=1}^{r-1}2^{a_j}\big).$$

\item \label{cor: count2}
Assume $\l>a_r$, say $a_p\geq \ell>a_{p+1}$.
Then
$$T_{0,\l}(n)=|\Omega(\ell;n)|=
\big((2^\ell-2)2^{a_p-\ell}\prod_{j=1}^{p-1}2^{a_j}
+2\prod_{j=1}^{p-1}(2^{a_j}-3)\big)
\prod_{j=p+1}^{r}2^{a_j}.$$

\item \label{cor: count3}
Assume $a_r= \ell$.
Then
$$T_{0,\l}(n)=
|\Omega(\ell;n)|=
2^{\ell+1}\big(\prod_{j=1}^{r-2}2^{a_j}
+\prod_{j=1}^{r-2}(2^{a_j}-3)\big).$$
\end{enumerate}
\end{corollary}

\begin{proof}
(1)
Using Corollary~\ref{cor: iteration counting} in an inductive argument,
we obtain the claimed formula.

(2)
The formula follows from Lemmas~\ref{lem: runner x not b0},
~\ref{lem: runner b0},~\ref{lem: runner b0-1}
and by Corollary~\ref{cor: iteration counting}.

(3)
Together with our previous results, we deduce
the claim using Lemma~\ref{lem: m=2 to ell}.
\end{proof}

\subsection*{Acknowledgements}
Part of this work was carried out while the second and third authors were visiting the
Institute of Algebra, Number Theory and Discrete Mathematics at Leibniz Universit\"at Hannover.
They would like to thank everyone at the institute for the kind hospitality.

\end{document}